\documentclass[10pt,a4paper]{amsart}
\usepackage{mathrsfs}
\usepackage{syntonly}
\usepackage{amsmath}
\usepackage{amsthm}
\usepackage{amsfonts}
\usepackage{amssymb}
\usepackage{latexsym}
\usepackage{amscd,amssymb,amsopn,amsmath,amsthm,graphics,amsfonts,mathrsfs,accents,enumerate,verbatim,calc}
\usepackage[dvips]{graphicx}
\usepackage[colorlinks=true,linkcolor=red,citecolor=blue]{hyperref}
\usepackage[all]{xy}
\usepackage{enumitem}
\usepackage{mathtools}
\usepackage{tikz}
\usetikzlibrary{decorations.pathreplacing,decorations.markings}
 \tikzset{
  on each segment/.style={
    decorate,
    decoration={
      show path construction,
      moveto code={},
      lineto code={
        \path [#1]
        (\tikzinputsegmentfirst) -- (\tikzinputsegmentlast);
      },
      curveto code={
        \path [#1] (\tikzinputsegmentfirst)
        .. controls
        (\tikzinputsegmentsupporta) and (\tikzinputsegmentsupportb)
        ..
        (\tikzinputsegmentlast);
      },
      closepath code={
        \path [#1]
        (\tikzinputsegmentfirst) -- (\tikzinputsegmentlast);
      },
    },
  },
  mid arrow/.style={postaction={decorate,decoration={
        markings,
        mark=at position 0.6 with {\arrow[#1]{stealth}} 
      }}},
}
\usetikzlibrary{arrows}
\usetikzlibrary{trees}

\usetikzlibrary{matrix}
\usetikzlibrary{patterns}
\usetikzlibrary{shadings} 
\date{}
\allowdisplaybreaks[4] \footskip=15pt
\renewcommand{\uppercasenonmath}[1]{}

\numberwithin{equation}{section} \theoremstyle{plain}
\newtheorem{theorem}{Theorem}[section]
\newtheorem{corollary}[theorem]{Corollary}
\newtheorem{lemma}[theorem]{Lemma}
\newtheorem{proposition}[theorem]{Proposition}
\theoremstyle{definition}
\newtheorem{definition}[theorem]{Definition}
\newtheorem{example}[theorem]{Example}
\newtheorem{remark}[theorem]{Remark}

\newtheorem*{ack*}{ACKNOWLEDGEMENTS}

\newcommand{\oo}{\otimes}

\newcommand{\pf}{\noindent\begin {proof}}
\newcommand{\epf}{\end{proof}}

\newcommand{\ra}{\rightarrow}

\newcommand{\Hom}{\mbox{\rm Hom}}

\newcommand{\im}{\mbox{\rm Im}}
\newcommand{\coker}{\mbox{\rm Coker}}
\newcommand{\ke}{\mbox{\rm Ker}}

\def\alg{\mathit{\Lambda}}
\def\pd{\mathsf{pd}}

\newcommand{\s}{\Omega}

	\newcommand{\BMod}{B\text{-}\mathrm{Mod}}

	\newcommand{\modA}{\mathrm{mod}\text{-}A}
	\newcommand{\modB}{\mathrm{mod}\text{-}B}

	\newcommand{\modL}{\mathrm{mod}\text{-}\alg}
    \newcommand{\modLe}{\mathrm{mod}\text{-}\alg/\alg e\alg}
    \newcommand{\modeLe}{\mathrm{mod}\text{-}e\alg e}

    \newcommand{\I}{\mathrm{Id}_{M}\oo}
    
	\def\add{\mathop{\rm add}\nolimits}

	\def\IT{\mathop{\rm IT.dist}\nolimits}
    
\begin{document}
\begin{center}
{{\bf\large
Igusa--Todorov properties of recollements of abelian categories\footnote{
Yu-Zhe Liu was partly supported by
the Scientic Research Foundation of Guizhou University (Grant Nos. [2022]53, [2022]65, [2023]16),
the Guizhou Provincial Basic Research Program (Natural Science) (Grant Nos. ZK[2024]YiBan066 and ZD[2025]085),
and the National Natural Science Foundation of China (Grant Nos. 12401042 and 12171207).
}}

\vspace{0.5cm}   Peiru Yang, Yajun Ma\footnote{{\color{white}\hspace{1mm}} Corresponding author.}, Yu-Zhe Liu}

\end{center}

$$\bf  Abstract$$
\leftskip0truemm \rightskip0truemm \noindent
In this paper, we investigate the behavior of Igusa--Todorov properties under recollements of abelian categories. In particular, we study how the Igusa--Todorov distances of the categories involved in a recollement are related.
Applications are given to Artin algebras, especially to Morita context rings.
\leftskip10truemm \rightskip10truemm \noindent
\\[2mm]

\leftskip0truemm \rightskip0truemm
\section{ \bf Introduction}

Let $\alg$ be an Artin algebra and let $\modL$ denote the category of finitely generated right $\alg$-modules.
Recall that the finitistic dimension $\mathsf{fin.dim}\alg$ of $\alg$ is defined as the supremum of the projective dimension of all finitely generated right $\alg$-modules of finite projective dimension.
The finitistic dimension conjecture asserts that $\mathsf{fin.dim}\alg<\infty$ for any Artin algebra $\alg.$

There esists a variety of literature on the study of the finitistic dimension of Artin algebras. Recently, Xi in \cite{xi2004finitistic,xi2006finitistic,xi2008finitistic}
developed various methods to detect finiteness of the finitistic dimension using Igusa--Todorov function introduced in \cite{IT}.
Based on Xi's work, Wei  introduced the concept of Igusa--Todorov algebras, and proved that the finitistic dimension conjecture holds for all Igusa--Todorov algebras in \cite{W}.
In particular, this class of algebras contains algebras with radical cube zero, monomial algebras, and algebras of representation dimension at most 3 among others.
However, T. Conde \cite{Conde} showed that not all Artin algebras are Igusa--Todorov algebras.
In order to measure how far an algebra is from being an Igusa--Todorov algebra, Zhang and Zheng \cite{ZZ1} introduced the concept of the Igusa--Todorov distance of Artin algebras, and studied its behavior under stable equivalences, singular equivalence of Morita type with level and recollements of derived module categories. Moreover, they proved that this distance is not only finite but also can bound the dimension of the singularity category.
On the other hand, Zheng and Huang \cite{Zheng-huang2022} gave an upper bound for the dimension of bounded derived categories of $n$-Igusa--Todorov algebras in terms of $n.$
To further study the dimension of derived categories, Zheng \cite{Z} introduced $(m,n)$-Igusa--Todorov algebras as a generalization of $n$-Igusa--Todorov algebras, and derived an upper bound for the dimension of bounded derived categories of $(m,n)$-Igusa--Todorov algebras in terms of $m$ and $n.$
We should note that the Igusa--Todorov distance of an Artin algebra is defined in terms of the concept of $(m,n)$-Igusa--Todorov algebra for some nonnegative integers $m$ and $n;$ see \cite[Definition 3.3]{ZZ1}.
Recently, Ma, Zhang and Liu \cite{MZL} also investigated Igusa--Todorov distances and the dimesion of bounded derived categories under the cleft extensions of abelian categories.

Recollements of abelian categories are exact sequences of abelian categories
$$0 \ra \mathscr{A}\xrightarrow{\sf i} \mathscr{B}\xrightarrow{\sf e}\mathscr{C}\ra 0$$
where both the inclusion functor $\mathsf{i}:\mathscr{A}\ra \mathscr{B}$ and the quotient functor $\mathsf {e}:\mathscr{B}\ra \mathscr{C}$ have left and right adjoints.
Recollements have been introduced first in the context of triangulated categories by Beilinson, Bernstein and Deligne \cite{B}, in their axiomatization of the Grothendieck's six functors for derived categories of sheaves obtained from the stratification of spaces.
It should be noted that recollements of abelian categories appear quite naturally in various settings and are omnipresent in representation theory, providing a proper framework for investigating the homological connections among these categories.
For example, Psaroudakis \cite{P} investigated homological properties of recollements of abelian categories, and then gave a series of bounds among the global, finitistic and representation dimensions of abelian categories linked by such recollements of abelian categories; Psaroudakis, Skarts{\ae}terhagen and Solberg \cite{PSS} studied singularity equivalences, Gorensteinness, and finite generation conditions in recollements of abelian categories; Zhang and Zhu \cite{Zhang} investigated the Gorenstein global dimension in recollements of abelian categories.

Inspired by this, the main purpose of this paper is to investigate the behavior of Igusa--Todorov properties under recollements of abelian categories.

Let $m$ and $n$ be nonnegative integers.
In order to study $(m,n)$-Igusa--Todorov algebras and Igusa--Todorov distances within a more general framework, Ma, Zheng and Liu introduced the concept of $(m,n)$-Igusa--Todorov categories in \cite{MZL}.
Let $\mathscr{A}$ be an abelian category with enough projectives.
Recall that abelian category $\mathscr{A}$ is called an $(m,n)$-Igusa--Todorov category if there exists an object $U\in \mathscr{A}$ such that for any $n$-syzygy object $A$, there exists an exact sequence
	$$0\ra U_{m}\ra U_{m-1}\ra\cdots\ra U_{1}\ra U_{0}\ra A\ra 0,$$
	with $U_{i}\in \add (U)$ for $0\leq i\leq m$.
 Let $\alg$ be an Artin algebra, and let $\mathrm{mod}$-$\alg$ be the category of finitely generated right $\alg$-modules.
 We note that $\modL$ is an $(m,n)$-Igusa--Todorov category if and only if $\alg$ is an $(m,n)$-Igusa--Todorov algebra in the sense of \cite[Definition 2.1]{Z}.
Our first result is the following, which is  listed as Theorem \ref{main theorem}.

\begin{theorem}\label{thm0}
	Let the following diagram be a recollement of abelian categories with $\mathsf{Proj}\mathscr{B}=\add(P)$ for some projective object $P$.
	$$\scalebox{0.85}{\xymatrixcolsep{2.0pc}\xymatrix{
			\mathscr{A}\ar[rr]^{\mathsf{i}} && \mathscr{B}\ar@/_2.0pc/[ll]_{\sf q}\ar@/^2.0pc/[ll]^{\sf p}\ar[rr]^{\sf e} && \mathscr{C} \ar@/_2.0pc/[ll]_{\sf l}\ar@/^2.0pc/[ll]^{\sf r} } }$$
Assume that $\mathscr{A}$ is an $(m,r)$-Igusa--Todorov category and $\mathscr{C}$ is an $(n,s)$-Igusa--Todorov category.
Then the following statements hold.
\begin{enumerate}
\item If both $\sf l$ and $\sf p$ are exact, then $\mathscr{B}$ is a $(2m+n+2,\max\{r,s\})$-Igusa--Todorov category.
\item If $\sf l, p$ and $\sf q$ are exact, then $\mathscr{B}$ is an $(m+n+1,\max\{r,s\})$-Igusa--Todorov category.
\end{enumerate}
\end{theorem}

We note that the concept of the Igusa--Todorov distance is defined in terms of $(m,n)$-Igusa--Todorov category for some nonnegative integers $m$ and $n$; see Definition \ref{def:IT distance}.
To further state our results precisely, we recall from \cite{P} the notion of the $\mathscr{A}$-relative global dimension of $\mathscr{B}$ , which is defined by
$$\mathsf{gl.dim}_{\mathscr{A}}\mathscr{B}:=\sup\{\mathsf{pd}_{\mathscr{B}}\mathsf{i}(A)\mid A\in \mathscr{A}\}.$$
Applying Theorem \ref{thm1} and Corollary \ref{corollary:IT}, we obtain a bound for the Igusa--Todorov distance of $\mathscr{B}$ in terms of the Igusa--Todorov distances of $\mathscr{A}$ and $\mathscr{C}$; see Corollary \ref{corollary:main}.

\begin{theorem}\label{thm2}
Let the following diagram be a recollement of abelian categories with $\mathsf{Proj}\mathscr{B}=\add(P)$ for some projective object $P$.
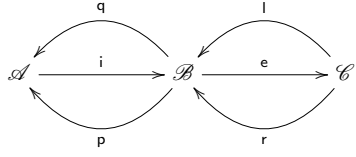
	$$\scalebox{0.85}{\xymatrixcolsep{2.0pc}\xymatrix{
			\mathscr{A}\ar[rr]^{\mathsf{i}} && \mathscr{B}\ar@/_2.0pc/[ll]_{\sf q}\ar@/^2.0pc/[ll]^{\sf p}\ar[rr]^{\sf e} && \mathscr{C} \ar@/_2.0pc/[ll]_{\sf l}\ar@/^2.0pc/[ll]^{\sf r} } }$$
 Assume that $\mathsf{l}$ is exact.
\begin{enumerate}
\item If $\mathsf{gl.dim}_{\mathscr{A}}\mathscr{B}<\infty,$ then
$$\IT(\mathscr{C})\leq\IT(\mathscr{B})\leq\IT(\mathscr{C})+2.$$
Moreover, if $\sf q$ is exact, then $$\IT(\mathscr{C})\leq\IT(\mathscr{B})\leq\IT(\mathscr{C})+1.$$
\item If  $\mathsf{p}$ is exact, then
$$\IT(\mathscr{C})\leq\IT(\mathscr{B})\leq 2\IT(\mathscr{A})+\IT(\mathscr{C})+2.$$
Moreover, if $\mathsf{q}$ is exact, then
$$\IT(\mathscr{C})\leq\IT(\mathscr{B})\leq \IT(\mathscr{A})+\IT(\mathscr{C})+1.$$
\end{enumerate}
\end{theorem}

We mention that there exists the recollement of module categories  $(\modLe,\modL,\modeLe)$ for any idempotent $e=e^{2}$ in an Artin algebra $\alg$.
Following \cite[Section 8]{P}, we denote by $\mathsf{gl.dim}_{\alg/\alg e\alg}\alg=\sup\{\mathsf{pd}X_{\alg}\mid X\in \mathrm{mod}\text{-}\alg/\alg e\alg\}.$
Hence as an application of Theorems \ref{thm0} and \ref{thm2}, we have the following result.

\begin{corollary}
Let $\alg$ be an Artin algebra and let $e=e^{2}$ be an idempotent in $\alg.$ Assume that $e\alg$ is a projective left $e\alg e$-module. Then the following statements hold.
\begin{enumerate}
\item Assume that $\alg/\alg e\alg$ is an $(m,r)$-Igusa--Todorov algebra and $e\alg e$ is an $(n,s)$-Igusa--Todorov algebra. If $\alg/\alg e\alg$ is a projective right $\alg$-module, then $\alg$ is a $(2m+n+2,\max\{r,s\})$-Igusa--Todorov algebra.
\item Assume that $\alg/\alg e\alg$ is an $(m,r)$-Igusa--Todorov algebra and $e\alg e$ is an $(n,s)$-Igusa--Todorov algebra. If $\alg/\alg e\alg$ is a projective left $\alg$-module and $\alg/\alg e\alg$ is a projective right $\alg$-module, then $\alg$ is an $(m+n+1,\max\{r,s\})$-Igusa--Todorov algebra.
\item Assume that $\alg/\alg e\alg$  and $e\alg e$ are syzygy-finite algebras.  If $\alg/\alg e\alg$ is a projective left $\alg$-module and $\alg/\alg e\alg$ is a projective right $\alg$-module, then $\alg$ is an Igusa--Todorov algebra.
\item If $\mathsf{gl.dim}_{\alg/\alg e\alg}\alg<\infty$, then
$$\IT(e\alg e)\leq\IT(\alg)\leq\IT(e\alg e)+2.$$
Moreover, if $\alg/\alg e\alg$ is a projective left $\alg$-module, then $$\IT(e\alg e)\leq\IT(\alg)\leq\IT(e\alg e)+1.$$
\item If $\alg/\alg e\alg$ is a projective right $\alg $-module, then
$$\IT(e\alg e)\leq\IT(\alg)\leq 2\IT(\alg/\alg e\alg)+\IT(e\alg e)+2.$$
 Moreover, if $\alg/\alg e\alg$ is a projective left $\alg$-module, then
$$\IT(e\alg e)\leq\IT(\alg)\leq \IT(\alg/\alg e\alg)+\IT(e\alg e)+1.$$
\end{enumerate}
\end{corollary}

The paper is structured as follows. In Section 2, we introduce essential notations and results of recollements of abelian categories, and give some examples that will be used in the sequel to illustrate our results. Section 3 is dedicated to the proofs of Theorem 1.1 and Theorem 1.2. In Section 4, we explore applications of these results to Morita context rings, and recollements of module categories arising from pairs $(\alg,e)$ consisting of an Artin algebra $\alg$ and an idempotent $e\in \alg.$

\section{ \bf Preliminaries}
\subsection{Conventions}
All algebras considered in this paper are Artin algebras and all modules are finitely generated.
For an Artin algebra $\alg$, we usually work with right $\alg$-modules and write $\modL$ for the category of finitely generated right $\alg$-modules.

All abelian categories considered are assumed to
	have enough projectives. Let $\mathscr{A}$ be an abelian category.
We denote by $\mathsf{Proj} \mathscr{A}$ the full subcategory of $\mathscr{A}$ consisting of projective objects.
For an object $U$ of $\mathscr{A}$, we use $\add (U)$ to
denote the subcategory of $\mathscr{A}$ consisting of
direct summands of finite direct sums of $U$.

\subsection{Recollements of abelian categories}\label{Section2.2}
We first recall the definition of a recollement of abelian categories, which was deeply investigated by Psaroudakis (\cite{P}).

\begin{definition}(\cite[Definition 2.1]{P})
	A {\bf recollement situation} between abelian category $\mathscr{A},\mathscr{B}$ and $\mathscr{C}$ is a diagram
	$$\scalebox{0.95}{\xymatrixcolsep{3pc}\xymatrix{
			\mathscr{A}\ar[rr]^{\mathsf{i}} && \mathscr{B}\ar@/_2.0pc/[ll]_{\sf q}\ar@/^2.0pc/[ll]^{\sf p}\ar[rr]^{\sf e} && \mathscr{C} \ar@/_2.0pc/[ll]_{\sf l}\ar@/^2.0pc/[ll]^{\sf r} } }$$
	satisfying the following:
	\begin{enumerate}
		\item $(\sf l,\sf e,\sf r)$ is an adjoint triple.
		\item $(\sf q,\sf i,\sf p)$ is an adjoint triple.
		\item The functors $\sf i,\sf l$ and $\sf r$ are fully faithful.
        \item $\im \sf i=\ke \sf e$
	\end{enumerate}
In this case, $\mathscr{B}$ is said to be the recollement of $\mathscr{A}$ and $\mathscr{B}$ and we denote the recollement by $(\mathscr{A},\mathscr{B},\mathscr{C})$.
\end{definition}

 In the following lemma, we  collect some basic properties of a recollement situation of abelian categories. For more details, we refer to \cite{P,PSS}.

\begin{lemma}{\rm (\cite[Proposition 2.2]{PSS}) }\label{lem of recollement}
	Let $(\mathscr{A},\mathscr{B},\mathscr{C})$ be a recollement of abelian categories. The following statements hold:
\begin{enumerate}
\item The functors $\mathsf{e}:\mathscr{B}\ra \mathscr{C}$ and $\mathsf{i}:\mathscr{A}\ra \mathscr{B}$ are exact.
\item The compositions $\sf ei, ql$ and $\sf pr$ are zero.
\item The unit $\mathrm{Id}_{\mathscr{C}}\xrightarrow{} \mathsf{el}$  of the adjoint pair $(\mathsf{l},\mathsf{e})$ and the counit $\mathsf{qi}\xrightarrow{}\mathrm{Id}_{\mathscr{A}}$ of the adjoint
pair $(\sf q, \sf i)$ are invertible.

\item The functors $\mathsf{l}$ and $\mathsf{q}$ preserve projective objects.
\item Let $B\in \mathscr{B}$. Then we have the exact sequence
$$0\ra \ke \mu_{B}\ra \mathsf{le}(B)\xrightarrow{\mu_{B}} B\xrightarrow{\alg_{B}} \mathsf{iq}(B)\ra 0$$
with $\ke \mu_{B}\in \sf i(\mathscr{A})$, where $\mu:\mathsf{le}\ra \mathrm{Id}_{\mathscr{B}}$ is the counit of $(\mathsf{l},\mathsf{e})$ and $\alg:\mathrm{Id}_{\mathscr{B}}\ra \mathsf{iq}$ is the unit of $(\mathsf{q},\mathsf{i}).$
		
\item Assume that $\mathsf{q}:\mathscr{B}\ra \mathscr{A}$ is exact.
Then for any $B\in \mathscr{B}$, there is an exact sequence
		$$0\ra \mathsf{le}(B)\xrightarrow{\mu_{B}} B\xrightarrow{\alg_{B}} \mathsf{iq}(B)\ra 0.$$
\end{enumerate}
\end{lemma}
\begin{proof}
We only give the proof of $(5)$.
By (4), we have the exact sequence
\begin{equation}\label{se.recollement}
0\ra \mathsf{i}(A)\ra \mathsf{le}(B)\xrightarrow{\mu_{B}} B\xrightarrow{\alg_{B}} \mathsf{iq}(B)\ra 0
\end{equation}
with some $A\in \mathscr{A}$.
Since $\mathsf{q}$ is exact by assumption, it follows that
the sequence
$$0\ra \mathsf{qi}(A)\ra \mathsf{qle}(B)\xrightarrow{\mathsf{q}(\mu_{B})} \mathsf{q}(B)\xrightarrow{\mathsf{q}(\alg_{B})} \mathsf{qiq}(B)\ra 0$$
is exact.
We mention that $\mathsf{ql}=0$ and $\mathsf{qi}\simeq \mathrm{Id}_{\mathscr{A}}$.
Thus $A=0$ and \eqref{se.recollement} becomes an exact sequence
$$0\ra \mathsf{le}(B)\xrightarrow{\mu_{B}} B\xrightarrow{\alg_{B}} \mathsf{iq}(B)\ra 0.$$
The proof is completed.
\end{proof}

\begin{example}(\cite[Example 2.3]{PSS})\label{idempotent}
Let $\alg$ be an Artin algebra and let $e=e^{2}$ be an idempotent in $\alg.$
Then we have the following recollement of abelian categories:
$$\scalebox{0.85}{\xymatrixcolsep{2pc}\xymatrix{
			\modLe\ar[rr]^{\mathrm{inc}} && \modL\ar@/_2.0pc/[ll]_{-\otimes_{\alg}\alg/\alg e\alg}\ar@/^2.0pc/[ll]^{\Hom_{\alg}(\alg/\alg e\alg,-)}\ar[rr]^{e(-)} && \modeLe \ar@/_2.0pc/[ll]_{-\otimes_{e\alg e}e\alg }\ar@/^2.0pc/[ll]^{\Hom_{e\alg e}(\alg e,-)} } }$$
Every $\alg/\alg e\alg$-module is annihilated by $\alg e\alg$ and thus the category $\modLe$ is the kernel of the functor $e(-).$
\end{example}

\begin{example}\label{ex:Morita ring}
Let $A$ and $B$ be two Artin algebras, $_{B}M_{A}$  a $B$-$A$-bimodule, $_{A}N_{B}$ an $A$-$B$-bimodule. Consider the Morita ring
$$\alg =\left(\begin{matrix}  A & {_{A}}N_{B}\\  {_{B}}M_{A} & B \\\end{matrix}\right)$$
with zero bimodule homomorphism, i.e. the multiplication is given by $$\left(\begin{matrix}  a & n \\  m & b \\\end{matrix}\right)\cdot\left(\begin{matrix}  a' & n' \\  m' & b' \\\end{matrix}\right)=\left(\begin{matrix}  aa' & an'+nb' \\  ma'+bm' & bb' \\\end{matrix}\right).$$

The description of modules over Morita rings $\alg$ is well known (see \cite{Gao,eg82}).
Let $\mathcal{M}(\alg)$ be the category whose objects are tuples $(X,Y,f,g),$ where $X\in\modA,~Y\in \modB,~f\in\Hom_{B}(X\oo_{A}N,Y)$ and $g\in\Hom_{A}(Y\oo_{B}M,X).$

A morphism $(X,Y,f,g)\ra(X',Y',f',g')$ is a pair $(a,b)$ of homomorphisms, where $a:X\ra X'$ is an $A$-homomorphism and $b:Y\ra Y'$ is a $B$-homomorphism such that the following diagrams are commutative:
 \begin{align*}
 \xymatrixcolsep{0.5pc}\xymatrix{
   X\oo_{A} N \ar[d]_{a\oo\mathrm{Id}_{N}}  \ar[rrr]^{f} & & &   Y \ar[d]^{b} & \\
  X'\oo_{A}N  \ar[rrr]^{f'} & & &  Y'}
  \xymatrixcolsep{0.5pc}\xymatrix{
   Y\oo_{B} M \ar[d]_{b\oo\mathrm{Id}_{M}}  \ar[rrr]^{g} & & &   X \ar[d]^{a} & \\
  Y'\oo_{B}M  \ar[rrr]^{g'} & & &  X'.}
\end{align*}
Then the categories $\modL$ and $\mathcal{M}(\alg)$ are equivalent.
For more details, we refer to \cite{GP}.
From now on we identify the modules over $\alg$ with the objects of $\mathcal{M}(\alg).$
Then there is a recollement of abelian categories
\begin{align*}
\scalebox{0.85}{\xymatrixcolsep{2pc}\xymatrix{
			\modA\ar[rr]^{\mathsf{Z}_{A}} && \modL\ar@/_2.0pc/[ll]_{\mathsf {C}_{A}}\ar@/^2.0pc/[ll]^{\mathsf{K}_{A}}\ar[rr]^{\mathsf {U}_{B}} && \modB \ar@/_2.0pc/[ll]_{\mathsf{T}_{B}}\ar@/^2.0pc/[ll]^{\mathsf{H}_{B}} } }
\end{align*}
where the functors are given as follows:
\begin{enumerate}
\item The functor $\mathsf{T}_{B}:\modB\ra \modL$ is defined on objects $Y\in \modB$ by $\mathsf{T}_{B}(Y)=(Y\otimes_{B}M,Y,0,1)$ and given a morphism $b:Y\ra Y'$ in $\modB$ then $\mathsf{T}_{B}(b)=(b\otimes \mathrm{Id}_{M},b)$ is a morphism in $\modL.$

\item The functor $\mathsf{U}_{B}:\modL\ra \modB$ is defined on objects  $(X,Y,f,g)$ by $\mathsf{U}_{B}(X,Y,f,g)=Y$ and given a morphism $(a,b):(X,Y,f,g)\ra (X',Y',f',g')$ in $\modL$ then $\mathsf{U}_{B}(a,b)=b.$

\item The functor $\mathsf{H}_{B}:\modB\ra \modL$ is defined on objects $Y\in \modB$ by $$\mathsf{H}_{B}(Y)=(\Hom_{B}(N,Y),Y,\varepsilon_{Y},0)$$ and given a morphism $b:Y\ra Y'$ in $\modB$ then $\mathsf{H}_{B}(b)=(\Hom_{B}(N,b),b)$ is a morphism in $\modL,$ where
    $\varepsilon_{}:\Hom_{B}(N,-)\otimes_{A}N\ra \mathrm{Id}_{\BMod}$
  is the counit of adjoint pair $(-\oo_{A}N,\Hom_{B}(N,-)).$

\item The functor $\mathsf{Z}_{A}:\modA\ra \modL$ is defined on objects $X\in \modA$ by $\mathsf{Z}_{A}(X)=(X,0,0,0)$ and given a morphism $a:X\ra X'$ in $\modA$ then $\mathsf{Z}_{A}(a):(X,0,0,0)\ra (X',0,0,0)$ is the induced morphism in $\modL$.
\item The functor $\mathsf{C}_{A}:\modL\ra \modA$ is defined on objects $(X,Y,f,g)\in \modL$ by $\mathsf{C}_{A}(X,Y,f,g)=\coker g$ and given a morphism $(a,b):(X,Y,f,g)\ra (X',Y',f',g')$ in $\modL$ then $\mathsf{C}_{A}(a,b):\coker g\ra \coker g'$ is the induced morphism in $\modA$.
\item The functor $\mathsf{K}_{A}:\modL\ra \modA$ is defined on objects $(X,Y,f,g)\in \modL$ by $\mathsf{K}_{A}(X,Y,f,g)=\ke \widetilde{f},$ where $\widetilde{f}:  X\to\Hom_{B}(N,Y)$ is defined as
    $\widetilde{f}(x)(n)=f(x\otimes n)$ for all $x\in X, n\in N$, and given a morphism $(a,b):(X,Y,f,g)\ra (X',Y',f',g')$ in $\modL$ then $\mathsf{K}_{A}(a,b):\ke \widetilde{g}\ra \ke \widetilde{g'}$ is the induced morphism in $\modA$.
\end{enumerate}
\end{example}

\section{\bf Igusa--Todorov properties}
In this section we first recall the notion of the Igusa--Todorov distance of an abelian category, which is a generalization of the Igusa--Todorov distance of an Artin algebra given in \cite{Z1}.

\begin{definition}\label{def:syzygy object}
Let $\mathscr{A}$ be an abelian category and $n$ a positive integer.
If there is an exact sequence
$$0\ra K\ra P_{n-1}\xrightarrow{d_{n-1}} P_{n-2}\xrightarrow{d_{n-2}} \cdots \ra P_{0}\xrightarrow{d_{0}} X\ra 0$$
in $\mathscr{A}$ such that each $P_{i}$ is  projective for $0\leq i\leq n-1$, then $K$ is called an $n$-syzygy object of $X$.
\end{definition}

\begin{remark}\label{rem:syzygy object}
For $X\in \mathscr{A}$, we know that its $n$-syzygy objects are not unique up to isomorphism, but are unique up to projective summands. Hence, for the sake of brevity,  we  denote by $\Omega^{n}(X)$ an $n$-syzygy object of $X$, and  denote by $\Omega^{n}(\mathscr{A})$ the subcategory of $\mathscr{A}$ consisting of all $n$-syzygy objects.
\end{remark}


\begin{definition}\label{def-IT-cat}(\cite[Definition 3.3]{MZL})
	Let $m$ and $n$ be nonnegative integers. Then an abelian category $\mathscr{A}$
	 is called an $(m,n)$-{\bf Igusa--Todorov category} if there exists an object $U\in \mathscr{A}$ such that for any $A\in \s^{n}(\mathscr{A})$, there is an exact sequence
	$$0\ra U_{m}\ra U_{m-1}\ra\cdots\ra U_{1}\ra U_{0}\ra A\ra 0,$$
	with $U_{i}\in \add (U)$ for $0\leq i\leq m$.
\end{definition}

\begin{definition}\label{def:IT distance}(\cite[Definition 3.4]{MZL})
	Let $\mathscr{A}$ be an abelian category.
	We set the {\bf Igusa--Todorov distance} of $\mathscr{A}$ as follows
	$$\IT(\mathscr{A}):=\inf\{m \mid\mathscr{A}\text{ is an}~ (m,n)\text{-Igusa--Todorov category for some nonnegative integer}~n\}.$$
\end{definition}

\begin{definition}(\cite[Definition 2.1]{Z})\label{(m,n)-IT}
Let $m$ and $n$ be nonnegative integers. An Artin algebra $\alg$ is called an {\bf $(m,n)$-Igusa--Todorov algebra} if there exists a $\alg$-module $V\in \modL$ such that  for any $M\in \modL$ there is an exact sequence
$$0\ra V_{m}\ra V_{m-1}\ra \cdots \ra V_{1}\ra V_{0}\ra \s^{n}(M)\ra 0$$
with $V_{i}\in \add(V)$ for each $0\leq i\leq m$.
\end{definition}

\begin{remark}
Let $\alg$ be an Artin algebra. Then $\modL$ is an $(m,n)$-Igusa--Todorov category if and only if $\alg$ is an an $(m,n)$-Igusa--Todorov algebra.
It follows that the Igusa--Todorov distance of $\modL$ equals the Igusa--Todorov distance of $\alg$; see \cite[definition 2.27]{Z1}.
\end{remark}

\begin{remark}\label{rem:IT}
Let $\alg$ be an Artin algebra and let $n$ be a nonnegative integer.
\begin{enumerate}
\item  $(m,n)$-Igusa--Todorov algebras are $(m,n+i)$-Igusa--Todorov algebras for any $i\geq 0$.
\item Recall from \cite[Definition 2.2]{W} that $\alg$ is called an {\bf $n$-Igusa--Todorov algebra} if there exists a $\alg$-module
$U$ such that for any $\alg$-module $X\in \Omega^{n}(\modL)$ there exists an exact sequence $0\ra U_{1}\ra U_{0}\ra X\ra 0$
with $U_{i}\in \mathrm{add}(U)$ for each $0\leq i \leq 1$.
 Furthermore, $\alg$ is called an {\bf Igusa--Todorov algebra} if it is an $n$-Igusa--Todorov algebra for some nonnegative integer $n.$
 By Definition \ref{(m,n)-IT}, we know that $(1,n)$-Igusa--Todorov algebras are precisely $n$-Igusa--Todorov algebras, and the $(0,n)$-Igusa--Todorov algebras are precisely {\bf $n$-syzygy finite algebras}, that is, $\Omega^{n}(\modL)$ is representation-finite for some nonnegative integer $n$.

\end{enumerate}
\end{remark}

\begin{remark}
Let $\alg$ be an Artin algebra.

\begin{enumerate}
\item By Definition \ref{def:IT distance}, $\alg$ is an $(m,n)$-Igusa--Todorov algebra if and only if $\IT(\modL)\leq m.$
In particular, $\alg$ is an Igusa--Todorov algebra  if and only if $\IT(\modL)\leq1.$
\item Recall that $\alg$ is called an {\bf syzygy-finite algebra} if $\alg$ is an $n$-syzygy-finite algebra for some nonnegative integer $n.$
     Then, by Definition \ref{def:IT distance}, $\alg$ is syzygy-finite if and only if $\IT(\modL)=0.$
\end{enumerate}
\end{remark}

\begin{example}{\rm
		Let $k$ be a field, and let $n\geq 1$ be an integer, and let
		$\alg$ be the quantum exterior algebra
		$$\alg=k[ X_{1},\cdots, X_{n}]/(X_{i}^{2},\{X_{i}X_{j}-q_{ij}X_{j}X_{i}\}_{i<j}),$$
		where $0\neq q_{ij}\in k$ and all the $q_{ij}$ are the roots of unity.
		We know from \cite[Example 2.35]{Z1} that $\IT(\modL)=n-1.$
	}
\end{example}

The following lemma can be found for Atrin algebras in \cite[Lemma 3.3]{XuD16}.
For the convenience of the reader, we provide the proof.
\begin{lemma} \label{IT-lem1}
	Let $\mathscr{A}$ be an abelian category. Then
	for any exact sequence
	$$0\longrightarrow X_{1} \longrightarrow  X_{2} \longrightarrow  X_{3} \longrightarrow 0$$
	in $\mathscr{A}$, we have the following exact sequence
	$$0\longrightarrow \s^{1}(X_{3})\longrightarrow X_{1}\oplus P \longrightarrow X_{2}\longrightarrow 0,$$
with  $P\in \mathsf{Proj}\mathscr{A}$.
\end{lemma}
\begin{proof}
 Since $\mathscr{A}$ has enough projectives, there exists an exact sequence
$$0\ra \s^{1}(X_{3})\ra P\xrightarrow{f} X_{3}\ra 0$$
 such that $P$ is projective.
 Consider the following pullback diagram:
 $$\xymatrix{&&0\ar[d]&0\ar[d]\\
		&&\s^{1}(X_{3})\ar@{=}[r]\ar[d]&\s^{1}(X_{3})\ar[d]\\
		0\ar[r]&X_{1}\ar[r]\ar@{=}[d]&Z\ar[d]\ar[r]&P\ar[d] \ar[r]&0\\
		0\ar[r] &X_{1}\ar[r]&X_{2}\ar[d]\ar[r]\ar[d]&X_{3}\ar[r]\ar[d]&0 \\&&0&0}$$
Since $P$ is projective, the exact sequence
$$0\ra X_{1}\ra Z\ra P\ra 0$$
splits and so $Z\cong X_{1}\oplus P$.
It implies that there exists an exact sequence
$$0\ra \s^{1}(X_{3})\ra X_{1}\oplus P \ra X_{2}\ra 0$$
with $P\in \mathsf{Proj}\mathscr{A}$, as desired.
\end{proof}

\begin{lemma}\label{lem-exact-syzygy}
	Let $\mathscr{A}$ and $\mathscr{B}$ be abelian categories and $n \geq 0$, and let $\mathsf{F}: \mathscr{A}\to \mathscr{B}$ be an exact
	functor preserving projective objects.
Then for any $X\in \mathscr{A}$ we have $$\mathsf{F}(\s^{n}(X))\oplus P'\cong \s^{n}(\mathsf{F}(X))\oplus P$$
for some projective objects $P$ and $P'$ in $\mathscr{B}.$
\end{lemma}
\begin{proof}
By Definition \ref{def:syzygy object} there exists an exact sequence
\begin{equation}\label{preserve syzygy}
0\ra \s^{n}(X)\ra P_{n-1}\ra P_{n-2}\ra \cdots \ra P_{0}\ra X\ra 0
\end{equation}
in $\mathscr{A}$ with each $P_{i}$ projective for $0\leq i\leq n-1$.
Since $\mathsf{F}$ is exact, applying $\sf F$ to \eqref{preserve syzygy}, we obtain an exact sequence
\begin{equation*}
0\ra \mathsf{F}(\s^{n}(X))\ra \mathsf{F}(P_{n-1})\ra \mathsf{F}(P_{n-2})\ra \cdots \ra \mathsf{F}(P_{0})\ra \mathsf{F}(X)\ra 0.
\end{equation*}
We mention that each $\mathsf{F}(P_{i})$ is projective for $0\leq i\leq n-1$ as $\mathsf{F}$ preserves projective objects.
It follows that $\mathsf{F}(\s^{n}(X))$ is an $n$-syzygy object of $\mathsf{F}(X)$.
Consequently, by Remark \ref{rem:syzygy object}, we have $\mathsf{F}(\s^{n}(X))\oplus P'\cong \s^{n}(\mathsf{F}(X))\oplus P$
for some projective objects $P$ and $P'$ in $\mathscr{B}.$
\end{proof}

The following lemma is key to the proof of Theorem \ref{main theorem}.
\begin{lemma}\label{IT-lem2}
Let  $0\ra K\ra  A\ra B\ra 0$  be an exact sequence in an abelian category $\mathscr{A}$.
	If there is an exact sequence
	$$0\ra V_{m}\xrightarrow{d_{m}}V_{m-1}\xrightarrow{d_{m-1}}\cdots\ra V_{1}\xrightarrow{d_{1}}V_{0}\xrightarrow{d_{0}}\Omega^{t}(A)\ra 0$$
	for some nonnegative integer $t$,
	then we have the following exact sequence
	$$0\ra \Omega^{t+m}(K)\ra V_{m}\oplus P_{m}\ra\cdots \ra V_{1}\oplus P_{1}\ra V_{0}\oplus P_{0}\ra \Omega^{t}(B)\ra 0,$$
	where each $P_{i}$ belongs to $\mathsf{Proj}\mathscr{A}$ for $0\leq i\leq m$.
\end{lemma}
\begin{proof}
	We denote $\ke d_{i}$ by $K_{i}$ for $0\leq i\leq m$.
	Consider the exact sequence
	\begin{equation}
		\label{se.3.1}
		0\ra K_{0}\ra V_{0}\ra \s^{t}(A)\ra 0.
	\end{equation}
Since $0\ra K\ra  A\ra B\ra 0$ is  exact, by the horseshoe Lemma and Remark \ref{rem:syzygy object}, we have the exact sequence
$$0\ra \s^{t}(K)\oplus Q_{0}\ra \s^{t}(A)\oplus P_{0}\ra \s^{t}(B)\ra 0$$
for some projective objects $P_{0}$ and $Q_{0}$.
Consider the following pullback diagram:
	$$\xymatrix{&0\ar[d]&0\ar[d]\\
		&K_{0}\ar@{=}[r]\ar[d]&K_{0}\ar[d]\\
		0\ar[r]&Z\ar[r]\ar[d]&V_{0}\oplus P_{0}\ar[d]\ar[r]&\s^{t}(B)\ar@{=}[d] \ar[r]&0\\
		0\ar[r] &\s^{t}(K)\oplus Q_{0}\ar[d]\ar[r]&\s^{t}(A)\oplus P_{0}\ar[d]\ar[r]\ar[d]&\s^{t}(B)\ar[r]&0 \\&0&0&}$$
Applying Lemma \ref{IT-lem1} to the left column in this diagram, we obtain the exact sequence
	$$0\ra \s^{t+1}(K)\ra K_{0}\oplus P_{1}\ra Z\ra 0$$ with $P_{1}\in \mathsf{Proj}\mathscr{A}$.
Next consider the exact sequence
	\begin{equation*}
		\label{se.3.2}
		0\ra K_{1} \ra V_{1}\ra K_{0}\ra 0
	\end{equation*}
	and the pullback diagram
	$$\xymatrix{&0\ar[d]&0\ar[d]\\
		&K_{1}\ar@{=}[r]\ar[d]&K_{1}\ar[d]\\
		0\ar[r]&M\ar[r]\ar[d]&V_{1}\oplus P_{1}\ar[d]\ar[r]&Z\ar@{=}[d] \ar[r]&0\\
		0\ar[r] &\s^{t+1}(K)\ar[d]\ar[r]&K_{0}\oplus P_{1}\ar[d]\ar[r]\ar[d]&Z\ar[r]&0 \\&0&0.&}$$
	Connecting the middle rows in above two diagrams together, one gets an exact sequence
	$$0\ra M\ra V_{1}\oplus P_{1}\ra V_{0}\oplus P_{0}\ra \s^{t}(B)\ra 0.$$
Continuing the process, we obtain exact sequences
	$$0\ra K_{m}\ra X\ra \s^{t+m}(K)\ra 0$$
	and
	$$0\ra X\ra V_{m}\oplus P_{m}\ra V_{m-1}\oplus P_{m-1}\ra \cdots\ra V_{0}\oplus P_{0}\ra \s^{t}(B)\ra 0$$
	with each $P_{i}$ projective for $0\leq i\leq m$.
	But $K_{m}=0$, thus $X\cong \s^{t+m}(K).$
	It follows that we obtain the  exact sequence
	$$0\ra\s^{t+m}(K)\ra V_{m}\oplus P_{m}\ra V_{m-1}\oplus P_{m-1}\ra\cdots \ra V_{0}\oplus P_{0}\ra \s^{t}(B)\ra 0,$$
	which completes the proof of the lemma.
\end{proof}

We begin with the following observation.
\begin{proposition}\label{prop1}
Let $(\mathscr{A},\mathscr{B},\mathscr{C})$ be a recollement of abelian categories such that $\sf l$ is exact.
Assume that there exists an exact sequence in $\mathscr{C}$ $$0\ra W_{n}\ra W_{n-1}\ra \cdots \ra W_{1}\ra W_{0}\ra \s^{t}(\mathsf e(B))\ra 0$$
for $B\in \mathscr{B}$ and some nonnegative integer $t$.
\begin{enumerate}
\item Then there is an exact sequence
\begin{equation}\label{IT.9}
 0\ra \mathsf{l}(W_{n})\ra\mathsf{l}(W_{n})\ra \cdots \ra \mathsf{l}(W_{1})\ra X\ra \s^{t}(\mathsf{le}(B))\ra 0
\end{equation}
with $X\in \add(\mathsf{l}(W_{0})\oplus Q)$ for some $Q\in \mathsf{Proj}\mathscr{B}$.
\item  Then there is an exact sequence
\begin{multline*}
0\ra \s^{t+n+2}(\mathsf {iq}(B))\ra \s^{t+n}(\mathsf i(A))\oplus Q_{n+1}\ra \mathsf{l}(W_{n})\oplus Q_{n}\ra \cdots \\
\ra \mathsf{l}(W_{1})\oplus Q_{1}\ra X\oplus Q_{0}\ra \s^{t}(B)\ra 0
\end{multline*}
for some $A\in \mathscr{A}$, where each $Q_{i}\in \mathsf{Proj}\mathscr{B}$ and $X\in \add(\mathsf{l}(W_{0})\oplus Q)$ for some $Q\in \mathsf{Proj}\mathscr{B}$.
\item If $\mathsf{q}$ is exact, then we have the exact sequence
$$0\ra \s^{t+n+1}(\mathsf{iq}(B))\ra\mathsf{l}(W_{n})\oplus P_{n}\ra \cdots \ra \mathsf{l}(W_{1})\oplus P_{1}\ra X\oplus P_{0}\ra \s^{t}(B)\ra 0,$$
where each $P_{i}\in \mathsf{Proj}\mathscr{B}$, and $X\in \add(\mathsf{l}(W_{0})\oplus Q)$ for some $Q\in \mathsf{Proj}\mathscr{B}$.
\end{enumerate}
\end{proposition}
\begin{proof}
(1)
By assumption, there exists an exact sequence
\begin{equation}\label{kse1}
0\ra W_{n}\ra W_{n-1}\ra \cdots \ra W_{1}\ra W_{0}\ra \s^{t}(\mathsf{e}(B))\ra 0.
\end{equation}
Since $\sf l$ preserves projective objects by Lemma \ref{lem of recollement} and is exact by assumption, it follows from Lemma \ref{lem-exact-syzygy}  that $$\s^{t}(\mathsf{le}(B))\oplus Q\cong \mathsf{l}(\s^{t}(\mathsf{e}(B)))\oplus P$$ for some projective objects $Q$ and $P$ in $\mathscr{B}$.
Hence applying $\sf l$ to the \eqref{kse1}, we get the exact sequence
\begin{equation}
		\label{IT.4}
		0\ra \mathsf{l}(W_{n})\ra\mathsf{l}(W_{n-1})\ra\cdots \ra \mathsf{l}(W_{1})\ra\mathsf{l}(W_{0})\oplus P\ra \s^{t}(\mathsf{le}(B))\oplus Q\ra 0.
	\end{equation}	
We denote by $K=\ke(\mathsf{l}(W_{0})\oplus P\ra \s^{t}(\mathsf{le}(B))\oplus Q)$.
Thus we get exact sequences
\begin{equation*}
0\ra K\ra \mathsf{l}(W_{0})\oplus P\ra \s^{t}(\mathsf{le}(B))\oplus Q\ra 0
\end{equation*}
and
\begin{equation}\label{IT.7}
0\ra \mathsf{l}(W_{n})\ra \mathsf{l}(W_{n-1})\ra \cdots \ra \mathsf{l}(W_{1})\ra K\ra 0.
\end{equation}
Consider the following pullback diagram:
$$\xymatrix{&&0\ar[d]&0\ar[d]\\
		0\ar[r]&K\ar@{=}[d]\ar[r]&X\ar[d]\ar[r]&\s^{t}(\mathsf{le}(B))\ar[d]\ar[r]&0\\
		0\ar[r]&K\ar[r]&\mathsf{l}(W_{0})\oplus P \ar[d]\ar[r]&\s^{t}(\mathsf{le}(B))\oplus Q\ar[d] \ar[r]&0\\
		 &&Q\ar[d]\ar@{=}[r]&Q\ar[d]&& \\&&0&0}$$
Thus we get an exact sequence
\begin{equation}\label{IT.8}
0\ra K\ra X\ra \s^{t}(\mathsf{le}(B))\ra 0
\end{equation}
 such that $X$ is a direct summand of $\mathsf{l}(W_{0})\oplus P$ as $Q$ is projective.
 Connecting exact sequences \eqref{IT.7} and \eqref{IT.8} together, we get the desired exact sequence
\begin{equation*}
 0\ra \mathsf{l}(W_{n})\ra\mathsf{l}(W_{n})\ra \cdots \ra \mathsf{l}(W_{1})\ra X\ra \s^{t}(\mathsf{le}(B))\ra 0
 \end{equation*}
 with $X\in \add(\mathsf{l}(W_{0})\oplus Q)$.

(2) By Lemma \ref{lem of recollement}(5), we have the exact sequence
$$0\ra \mathsf{i}(A)\ra \mathsf{le}(B)\xrightarrow{\mu_{B}} B\xrightarrow{\alg_{B}} \mathsf{iq}(B)\ra 0$$
for some $A\in \mathscr{A}$. Consider the exact sequence
\begin{equation}\label{IT.10}
0\ra \mathsf{i}(A)\ra \mathsf{le}(B)\ra \ke\alg_{B}\ra 0.
\end{equation}
Applying Lemma \ref{IT-lem2} to \eqref{IT.9} and \eqref{IT.10}, we obtain the following exact sequence
\begin{equation}\label{IT.11}
0\ra \s^{t+n}(\mathsf{i}(A))\ra  \mathsf{l}(W_{n})\oplus P_{n}\ra \cdots \ra \mathsf{l}(W_{1})\oplus P_{1}\ra X\oplus P_{0}\ra \s^{t}(\ke\alg_{B})\ra 0
\end{equation}
with each $P_{i}\in \mathsf{Proj}\mathscr{B}.$
On the other hand, using Lemma \ref{IT-lem1} to the exact sequence
\begin{equation*}
0\ra \ke\alg_{B}\ra B\ra\mathsf{iq}(B)\ra 0,
\end{equation*}
we obtain an exact sequence
\begin{equation}\label{IT.12}
0\ra \s(\mathsf{iq}(B))\ra \ke\alg_{B}\oplus P'\ra B\ra 0
\end{equation}
with $P'\in \mathsf{Proj}\mathscr{B}$.
Using Lemma \ref{IT-lem2} to \eqref{IT.11} and \eqref{IT.12} again, we have the following exact sequence
\begin{multline}
0\ra\s^{t+n+2}(\mathsf{iq}(B))\ra \s^{t+n}(\mathsf{i}(A))\oplus Q_{n+1}\ra\\  \mathsf{l}(W_{n})\oplus P_{n}\oplus Q'_{n}\ra \cdots \ra \mathsf{l}(W_{1})\oplus P_{1}\oplus Q'_{1}\ra X\oplus P_{0}\oplus Q'_{0}\ra \s^{t}(B)\ra 0
\end{multline}
with each $Q'_{i}\in \mathsf{Proj}\mathscr{B}$ for $0\leq i\leq n+1$ and $X\in \add(\mathsf{l}(W_{0})\oplus Q)$ for some $Q\in \mathsf{Proj}\mathscr{B}$.
Set $Q_{i}=P_{i}\oplus Q'_{i}$ for $0\leq i\leq n$. Thus we get the desired exact sequence.

(3) Since  $\mathsf{q}$ is exact, it follows from Lemma \ref{lem of recollement}(6) that there exists an exact sequence
$$0\ra \mathsf{le}(B)\xrightarrow{\mu_{B}} B\xrightarrow{\alg_{B}} \mathsf{iq}(B)\ra 0.$$
By Lemma \ref{IT-lem1}, we obtain an exact sequence
\begin{equation}\label{IT.21}
0\ra \s(\mathsf{iq}(B))\ra \mathsf{le}(B)\oplus P\ra B\ra 0
\end{equation}
with $P\in \mathsf{Proj}\mathscr{B}$.
By (1), we know that there is an exact sequence
\begin{equation}\label{IT.22}
 0\ra \mathsf{l}(W_{n})\ra\mathsf{l}(W_{n})\ra \cdots \ra \mathsf{l}(W_{1})\ra X\ra \s^{t}(\mathsf{le}(B))\ra 0
\end{equation}
with $X\in \add(\mathsf{l}(W_{0})\oplus Q)$ for some $Q\in \mathsf{Proj}\mathscr{B}$.
Using Lemma \ref{IT-lem2} to \eqref{IT.21} and \eqref{IT.22} again, we have the following exact sequence
\begin{equation}
 0\ra \s^{t+n+1}(\mathsf{iq}(B))\ra\mathsf{l}(W_{n})\oplus P_{n}\ra \cdots \ra \mathsf{l}(W_{1})\oplus P_{1}\ra X\oplus P_{0}\ra \s^{t}(B)\ra 0
\end{equation}
with each $P_{i}\in \mathsf{Proj}\mathscr{B}$ and $X\in \add(\mathsf{l}(W_{0})\oplus Q)$.
\end{proof}

To proceed further we recall from \cite{P} the notion of the {\bf $\mathscr{A}$-relative global dimension} of $\mathscr{B}$ , which is defined by
$$\mathsf{gl.dim}_{\mathscr{A}}\mathscr{B}:=\sup\{\mathsf{pd}_{\mathscr{B}}\mathsf{i}(A)\mid A\in \mathscr{A}\}.$$
\begin{proposition}\label{thm1}
Let $(\mathscr{A},\mathscr{B},\mathscr{C})$ be a recollement of abelian categories with $\mathsf{Proj}\mathscr{B}=\add(P)$ for some projective object $P$.
 Assume that $\sf l$ is exact and $\mathscr{C}$ is an $(n,t)$-Igusa--Todorov category.
Then the following statements hold.
\begin{enumerate}
\item If $n+t<\mathsf{gl.dim}_{\mathscr{A}}\mathscr{B}<\infty$, then $\mathscr{B}$ is a $(\mathsf{gl.dim}_{\mathscr{A}}\mathscr{B}-t+1,t)$-Igusa--Todorov category.
\item If $n+t\geq \mathsf{gl.dim}_{\mathscr{A}}\mathscr{B}$, then $\mathscr{B}$ is an $(n+2,t)$-Igusa--Todorov category.
\item Assume that $\sf q$ is exact. If $n+t+1<\mathsf{gl.dim}_{\mathscr{A}}\mathscr{B}<\infty$, then $\mathscr{B}$ is a $(\mathsf{gl.dim}_{\mathscr{A}}\mathscr{B}-t,t)$-Igusa--Todorov category.
\item Assume that $\sf q$ is exact.
If $n+t+1\geq \mathsf{gl.dim}_{\mathscr{A}}\mathscr{B}$, then $\mathscr{B}$ is an $(n+1,t)$-Igusa--Todorov category.
\end{enumerate}
\end{proposition}
\begin{proof}
Since $\mathscr{C}$ is an $(n,t)$-Igusa--Todorov category, it follows that for any $B\in \mathscr{B}$ there exist an object $W\in \mathscr{C}$ and an exact sequence
\begin{equation*}
0\ra W_{n}\ra W_{n-1}\ra \cdots \ra W_{0}\ra \s^{t}(\mathsf{e}(B))\ra 0
\end{equation*}
with $W_{i}\in \add(W)$ for $0\leq i\leq n$.
Then by Proposition \ref{prop1}, we obtain the following exact sequence
\begin{equation}\label{IT.17}
0\ra \s^{t+n+2}(\mathsf {iq}(B))\ra \s^{t+n}(\mathsf i(A))\oplus Q_{n+1}\ra \mathsf{l}(W_{n})\oplus Q_{n}\ra \cdots \ra X\oplus Q_{0}\ra \s^{t}(B)\ra 0
\end{equation}
for some $A\in \mathscr{A}$, where each $Q_{i}\in \mathsf{Proj}\mathscr{B}$ and $X\in \add(\mathsf{l}(W)\oplus P)$.
We denote by $C$ the cokernel of $\s^{t+n+2}(\mathsf{iq}(B))\ra \s^{t+n}(\mathsf{i}(A))\oplus Q_{n+1}.$
Thus we have the exact sequence
$$0\ra \s^{t+n+2}(\mathsf{iq}(B))\ra \s^{t+n}(\mathsf{i}(A))\oplus Q_{n+1}\ra C\ra 0.$$
We know from \cite[Lemma 4.2]{P} that
$$\mathsf{pd}_{\mathscr{B}}C\leq \mathsf{max}\{\mathsf{pd}_{\mathscr{B}}\s^{t+n+2}(\mathsf{iq}(B))+1,\mathsf{pd}_{\mathscr{B}}\s^{t+n}(\mathsf{i}(A))\}.$$

(1) Assume that $n+t<\mathsf{gl.dim}_{\mathscr{A}}\mathscr{B}<\infty$.
Then
$$\mathsf{pd}_{\mathscr{B}}C\leq \mathsf{max}\{\mathsf{pd}_{\mathscr{B}}\s^{t+n+2}(\mathsf{iq}(B))+1,\mathsf{pd}_{\mathscr{B}}\s^{t+n}(\mathsf{i}(A))\}
\leq\mathsf{gl.dim}_{\mathscr{A}}\mathscr{B}-(t+n).$$
We put $\mathsf{gl.dim}_{\mathscr{A}}\mathscr{B}=k$.
It follows that there is an exact sequence
$$0\ra P_{k-(t+n)}\ra \cdots \ra P_{1}\ra P_{0}\ra C\ra 0$$
with each $P_{i}\in \mathsf{Proj}\mathscr{B},$
which gives the following exact sequence
\begin{equation}\label{IT.20}
\scalebox{1}{$
0\ra P_{k-(t+n)}\ra\cdots \ra P_{1} \ra P_{0}\ra  \mathsf{l}(W_{n})\oplus Q_{n}\ra \cdots \ra \mathsf{l}(W_{1})\oplus Q_{1}\ra X\oplus Q_{0}\ra \s^{t}(B)\ra 0$}.
\end{equation}
Since $\mathsf{Proj}\mathscr{B}=\add(P)$ for some projective object $P$ and $X\in \add(\mathsf{l}(W)\oplus P)$, it follows that each term in \eqref{IT.20}, with the exception of $\s^{t}(B)$, belongs to $\add(P\oplus \mathsf{l}(W)).$
It implies that $\mathscr{B}$ is a $(\mathsf{gl.dim}_{\mathscr{A}}\mathscr{B}-t+1,t)$-Igusa--Todorov category.

(2) If $\mathsf{gl.dim}_{\mathscr{A}}\mathscr{B}\leq t+n,$ then $\s^{t+n+2}(\mathsf{iq}(B))$ and $\s^{t+n}(\mathsf{i}(A))$ are projective.
Since $\mathsf{Proj}\mathscr{B}=\add(P)$ by assumption, it follows that $\s^{t+n+2}(\mathsf{iq}(B))$ and $\s^{t+n}(\mathsf{i}(A))$ belong to $\add(P).$
Hence each term in \eqref{IT.17}, with the exception of $\s^{t}(B)$, belongs to $\add(\mathsf{l}(W)\oplus P)$. Hence $\mathscr{B}$ is an $(n+2,t)$-Igusa--Todorov category.

The assertions (3) and (4) can be similarly proved by using Proposition \ref{prop1}(3).
\end{proof}

\begin{corollary}\label{corollary:IT}
Let $(\mathscr{A},\mathscr{B},\mathscr{C})$ be a recollement of abelian categories with $\mathsf{Proj}\mathscr{B}=\add(P)$ for some projective object $P$ with $\mathsf{gl.dim}_{\mathscr{A}}\mathscr{B}<\infty$.
\begin{enumerate}
\item If $\sf l$ is exact, then
$\IT(\mathscr{B})\leq \IT(\mathscr{C
})+2.$
\item If both $\sf l$ and $\sf q$ are exact, then
$\IT(\mathscr{B})\leq \IT(\mathscr{C
})+1.$
\end{enumerate}
\end{corollary}
\begin{proof}
We only prove (1).
Assume that $\IT(\mathscr{C})=n$.
Then $\mathscr{C}$ is an $(n,t)$-Igusa--Todorov category for some nonnegative integer $t.$
By Remark \ref{rem:IT}(3), we may assume that $n+t\geq \mathsf{gl.dim}_{\mathscr{A}}\mathscr{B}.$
Hence $\mathscr{B}$ is an $(n+2,t)$-Igusa--Todorov category, which implies that $\IT(\mathscr{B})\leq \IT(\mathscr{C
})+2.$
\end{proof}

The following results describe the Igusa--Todorov properties of $\mathscr{B}$ in terms of those of $\mathscr{A}$ and $\mathscr{C}.$
\begin{theorem}\label{main theorem}
Let $(\mathscr{A},\mathscr{B},\mathscr{C})$ be a recollement of abelian categories with $\mathsf{Proj}\mathscr{B}=\add(P)$ for some projective object $P$. Assume that $\mathscr{A}$ is an $(m,r)$-Igusa--Todorov category and $\mathscr{C}$ is an $(n,s)$-Igusa--Todorov category. Then the following statements hold.
\begin{enumerate}
\item If both $\sf l$ and $\sf p$ are exact, then $\mathscr{B}$ is a $(2m+n+2,\max\{r,s\})$-Igusa--Todorov category.
\item If $\sf l, p$ and $\sf q$ are exact, then $\mathscr{B}$ is an $(m+n+1,\max\{r,s\})$-Igusa--Todorov category.
\end{enumerate}
\end{theorem}
\begin{proof}
We set $t=\max\{r,s\}$.
Then $\mathscr{A}$ and $\mathscr{C}$ can be seen as an $(m,t)$-Igusa--Todorov category and an $(n,t)$-Igusa--Todorov category respectively by Remark \ref{rem:IT}.
Thus for any $B\in \mathscr{B}$ there exist an object $W\in \mathscr{C}$ and an exact sequence in $\mathscr{C}$
$$0\ra W_{n}\ra W_{n-1}\ra \cdots \ra W_{0}\ra \s^{t}(\mathsf e(B))\ra 0$$
with each $W_{i}\in \add(W).$

(1) By Proposition \ref{prop1}(2), there exists an exact sequence
$$0\ra \s^{t+n+2}(\mathsf {iq}(B))\ra \s^{t+n}(\mathsf i(A))\oplus Q_{n+1}\ra \mathsf{l}(W_{n})\oplus Q_{n}\ra \cdots \ra X\oplus Q_{0}\ra \s^{t}(B)\ra 0$$
for some $A\in \mathscr{A}$, where each $Q_{i}\in \mathsf{Proj}\mathscr{B}$ and $X\in \add(\mathsf{l}(W)\oplus P)$.
We denote by $C$ the cokernel of $\s^{t+n+2}(\mathsf{iq}(B))\ra \s^{t+n}(\mathsf{i}(A))\oplus Q_{n+1}$. Then we have exact sequences
\begin{equation}\label{IT.15}
0\ra\s^{t+n+2}(\mathsf{iq}(B))\ra \s^{t+n}(\mathsf{i}(A))\oplus Q_{n+1}\ra C\ra 0.
\end{equation}
and
\begin{equation*}
0\ra C\ra \mathsf{l}(W_{n})\oplus Q_{n}\ra \cdots \ra X\oplus Q_{0}\ra \s^{t}(B)\ra 0.
\end{equation*}
We mention that $\sf p$ is exact by assumption and $(\mathsf{i},\mathsf{p})$ is an adjoint pair. Hence $\mathsf{i}$ preserves projectives.
Since $\sf i$ is exact, it follows from Lemma \ref{lem-exact-syzygy} that
$$\s^{t+n}(\mathsf{i}(A))\oplus P'\cong  \mathsf{i}(\s^{t+n}(A))\oplus P''$$
for some projective objects $P'$ and $P''$ in $\mathscr{B}$.
Note that $\mathscr{A}$ is an $(m,t)$-Igusa--Todorov category by assumption.
Then there exist an object $V\in \mathscr{A}$ and an exact sequence
\begin{equation}\label{IT.13}
0\ra V_{m}\ra \cdots\ra V_{1}\ra V_{0}\ra \s^{t+n}(A)\ra 0
\end{equation}
with each $V_{i}\in \add(V)$ for $0\leq i\leq m$.
Applying $\sf i$ to \eqref{IT.13} we obtain the following exact sequence
\begin{equation}\label{IT.18}
0\ra \mathsf{i}(V_{m})\ra \cdots\ra \mathsf{i}(V_{1})\ra \mathsf{i}(V_{0})\oplus P''\ra \s^{t+n}(\mathsf{i}(A))\oplus P'\ra 0.
\end{equation}
Note that we have the exact sequence
\begin{equation}\label{IT.15}
0\ra\s^{t+n+2}(\mathsf{iq}(B))\oplus P'\ra \s^{t+n}(\mathsf{i}(A))\oplus Q_{n+1}\oplus P'\ra C\ra 0.
\end{equation}
Applying Lemma \ref{IT-lem2} to \eqref{IT.18} and \eqref{IT.15}, we have the exact sequence
$$0\ra \s^{t+n+m+2}(\mathsf{iq}(B))\ra \mathsf{i}(V_{m})\oplus P_{m}\ra \cdots \ra \mathsf{i}(V_{0})\oplus P_{0}\ra C\ra 0$$
with each $P_{i}$ belonging to $\mathsf{Proj}\mathscr{B}.$
On the other hand, since $\sf i$ is exact and preserves projectives, it follows from Lemma \ref{lem-exact-syzygy} that
 $$\s^{t+n+m+2}(\mathsf{iq}(B))\oplus Q\cong \mathsf{i}(\s^{t+n+m+2}(\mathsf{q}(B)))\oplus Q'$$
 for some projective objects $Q$ and $Q'$ in $\mathscr{B}.$
Since $\mathscr{A}$ is an $(m,t)$-Igusa--Todorov category, it follows that there exists an exact sequence
\begin{equation}\label{IT.16}
0\ra V'_{m}\ra \cdots \ra V'_{1}\ra V'_{0}\ra \s^{t+n+m+2}(\mathsf{q}(B))\ra 0
\end{equation}
with each $V'_{i}\in \add(V)$.
Applying $\mathsf{i}$ to \eqref{IT.16} we obtain the follwing exact sequence
\begin{equation}\label{IT.19}
0\ra \mathsf{i}(V'_{m})\ra \cdots \ra \mathsf{i}(V'_{1})\ra \mathsf{i}(V'_{0})\oplus Q'\ra \s^{t+n+m+2}(\mathsf{iq}(B))\oplus Q\ra 0.
\end{equation}
Therefore we have the exact sequence
\begin{multline*}
0\ra \mathsf{i}(V'_{m})\ra \cdots\ra \mathsf{i}(V'_{1})\ra \mathsf{i}(V'_{0})\oplus Q'\\
\ra \mathsf{i}(V_{m})\oplus P_{m}\oplus Q\ra \cdots \ra \mathsf{i}(V_{0})\oplus P_{0}\ra \mathsf{l}(W_{n})\oplus Q_{n}\ra \cdots \ra \mathsf{l}(W_{0})\oplus Q_{0}\ra \s^{t}(B)\ra 0.
\end{multline*}
Note that $\mathsf{Proj}\mathscr{B}=\add(P)$ for some projective object $P$ by assumption.
Therefore we conclude that $\mathscr{B}$ is a $(2m+n+2,\max\{r,s\})$-Igusa--Todorov category.

(2) By Proposition \ref{prop1}(3), we have the exact sequence
\begin{equation}\label{IT.25}
0\ra \s^{t+n+1}(\mathsf{iq}(B))\ra\mathsf{l}(W_{n})\oplus P_{n}\ra \cdots \ra \mathsf{l}(W_{1})\oplus P_{1}\ra X\oplus P_{0}\ra \s^{t}(B)\ra 0
\end{equation}
with $X\in \add(\mathsf{l}(W_{0})\oplus P)$ and each $P_{i}\in \mathsf{Proj}\mathscr{B}$.
Since $\sf p$ is exact and $(\sf i, p)$ is an adjoint pair, it follows that $\sf i$ preserves projective objects.
Thus by Lemma \ref{lem-exact-syzygy}, we have an isomorphism $$\mathsf{i}(\s^{t+n+1}(\mathsf{q}(B)))\oplus Q'\cong \s^{t+n+1}(\mathsf{iq}(B))\oplus Q$$ for some projective objects $Q$ and $Q'$ in $\mathscr{B}$.
We mention that $\mathscr{A}$ is an $(m,t)$-Igusa--Todorov category. Thus
there exist an object $V\in \mathscr{A}$ and an exact sequence
\begin{equation}\label{IT.23}
0\ra V'_{m}\ra \cdots\ra V'_{1}\ra V'_{0}\ra \s^{t+n+1}(\mathsf{q}(B))\ra 0
\end{equation}
with each $V'_{i}\in \add(V)$ for $0\leq i\leq m$. Applying $\sf i$ to \eqref{IT.23}, we get the following exact sequence
\begin{equation}\label{IT.24}
0\ra \mathsf{i}(V'_{m})\ra \cdots\ra \mathsf{i}(V'_{1})\ra \mathsf{i}(V'_{0})\oplus Q'\ra \s^{t+n+1}(\mathsf{iq}(B))\oplus Q\ra 0.
\end{equation}
Connecting \eqref{IT.25} and \eqref{IT.24} together, we obtain the exact sequence
\begin{equation}\label{IT.26}
0\ra \mathsf{i}(V'_{m})\ra\cdots\ra\mathsf{i}(V'_{0})\oplus Q'\ra\mathsf{l}(W_{n})\oplus P_{n}\oplus Q\ra \cdots \ra \mathsf{l}(W_{1})\oplus P_{1}\ra X\oplus P_{0}\ra \s^{t}(B)\ra 0.
\end{equation}
Since $\mathsf{Proj}\mathscr{B}=\add(P)$ for some projective object $P$, it follows that each term in \eqref{IT.26}, with the exception of $\s^{t}(B)$, belongs to $\add(\mathsf{l}(W)\oplus \mathsf{i}(V)\oplus P)$. Consequently $\mathscr{B}$ is an $(n+m+1,t)$-Igusa--Todorov category.
\end{proof}

\begin{corollary}\label{main theorem1}
Let $(\mathscr{A},\mathscr{B},\mathscr{C})$ be a recollement of abelian categories with $\mathsf{Proj}\mathscr{B}=\add(P)$ for some projective object $P$.
Then the following statements hold.
\begin{enumerate}
\item Assume that both $\sf l$ and $\sf p$ are exact.
Then $$\IT(\mathscr{B})\leq2\IT(\mathscr{A})+\IT(\mathscr{C})+2.$$
\item Assume that $\sf l, q$ and $\sf p$ are exact. Then $$\IT(\mathscr{B})\leq\IT(\mathscr{A})+\IT(\mathscr{C})+1.$$
\end{enumerate}
\end{corollary}
\begin{proof}
We may assume that $\IT(\mathscr{A})=m<\infty$ and $\IT(\mathscr{C})=n<\infty$.
Then $\mathscr{A}$ is an $(m,r)$-Igusa--Todorov category and $\mathscr{C}$ is an $(n,s)$-Igusa--Todorov category for some $r$ and $s.$

(1) By Proposition \ref{main theorem}(1), $\mathscr{B}$ is a $(2m+n+2,\max\{r,s\})$-Igusa--Todorov category.
Hence $\IT(\mathscr{B})\leq2\IT(\mathscr{A})+\IT(\mathscr{C})+2.$

(2) By Proposition \ref{main theorem}(2), $\mathscr{B}$ is a $(m+n+1,\max\{r,s\})$-Igusa--Todorov category.
Hence $\IT(\mathscr{B})\leq\IT(\mathscr{A})+\IT(\mathscr{C})+1.$
\end{proof}

\begin{proposition}\label{prop2}
Let $(\mathscr{A},\mathscr{B},\mathscr{C})$ be a recollement of abelian categories with $\mathsf{Proj}\mathscr{B}=\add(P)$ for some projective object. If $\mathsf{l}$ is exact,
then $\IT(\mathscr{C})\leq\IT(\mathscr{B}).$
\end{proposition}
\begin{proof}
We may assume that $\IT(\mathscr{B})=m<\infty$ and $\mathscr{B}$ is an $(m,k)$-Igusa--Todorov category for some $k$.
Then for any $C\in \mathscr{C}$ there exist an object $U\in\mathscr{B}$ and an exact sequence in $\mathscr{B}$
\begin{equation}
0\ra U_{m}\ra \cdots \ra U_{1}\ra U_{0}\ra \s^{k}(\mathsf{l}(C))\ra 0
\end{equation}
with $U_{i}\in \add(U)$ for $0\leq i\leq m$.
Since $\sf l$ is exact and preserves projectives, it follows from Lemma \ref{lem-exact-syzygy} that
 $\s^{k}(\mathsf{l}(C))\oplus P'\cong\mathsf{l}(\s^{k}(C))\oplus P''$ for some projective objects $P'$ and $P''$ in $\mathscr{B}$.
Thus we have the following exact sequence
\begin{equation}\label{IT.27}
0\ra U_{m}\ra \cdots \ra U_{1}\ra U_{0}\oplus P'\ra \mathsf{l}(\s^{k}(C))\oplus P''\ra 0.
\end{equation}
We denote by $K$ the kernel of $U_{0}\oplus P'\ra \mathsf{l}(\s^{k}(C))\oplus P''$.
Then we obtain exact sequences
\begin{equation}\label{IT.28}
0\ra U_{m}\ra \cdots \ra U_{1}\ra K\ra 0
\end{equation}
and
\begin{equation*}
0\ra K\ra U_{0}\oplus P'\ra \mathsf{l}(\s^{k}(C))\oplus P''\ra 0.
\end{equation*}
Consider the following pullback diagram:
$$\xymatrix{&&0\ar[d]&0\ar[d]\\
		0\ar[r]&K\ar@{=}[d]\ar[r]&X\ar[d]\ar[r]&\mathsf{l}(\s^{k}(C)\ar[d]\ar[r]&0\\
		0\ar[r]&K\ar[r]&U_{0}\oplus P' \ar[d]\ar[r]&\mathsf{l}(\s^{k}(C))\oplus P''\ar[d] \ar[r]&0\\
		 &&P''\ar[d]\ar@{=}[r]&P''\ar[d]&& \\&&0&0}$$
Thus we obtain the exact sequence
\begin{equation}\label{IT.29}
0\ra K\ra X\ra \mathsf{l}(\s^{k}(C))\ra 0
\end{equation}
such that $X$ is a direct summand of $\mathsf{l}(U_{0})\oplus P'$ as $P''$ is projective.
Splicing \eqref{IT.28} with \eqref{IT.29} together, we obtain the following exact sequence
\begin{equation}\label{IT.30}
0\ra U_{m}\ra \cdots \ra U_{1}\ra X\ra \mathsf{l}(\s^{k}(C))\ra 0.
\end{equation}
Since $\sf e$ is exact, applying $\mathsf{e}$ to \eqref{IT.30}, we get the following exact sequence in $\mathscr{C}$
\begin{equation}\label{IT.31}
0\ra \mathsf{e}(U_{m})\ra \cdots \ra \mathsf{e}(U_{1})\ra \mathsf{e}(X)\ra \s^{k}(C)\ra 0.
\end{equation}
Since $\mathsf{Proj}\mathscr{B}=\add(P)$ and $X$ is a direct summand of $\mathsf{l}(U_{0})\oplus P'$, it follows that each term in \eqref{IT.31}, with the exception of $\s^{k}(C)$, belongs to $\add(\mathsf{e}(U\oplus P))$.
Therefore, we conclude that $\IT(\mathscr{C})\leq\IT(\mathscr{B}).$
\end{proof}

As a consequence of Corollary \ref{corollary:IT}, Theorem \ref{main theorem1} and  Proposition \ref{prop2}, we have the following.
\begin{corollary}\label{corollary:main}
Let $(\mathscr{A},\mathscr{B},\mathscr{C})$ be a recollement of abelian categories with $\mathsf{Proj}\mathscr{B}=\add(P)$ for some projective object. Assume that $\mathsf{l}$ is exact.
\begin{enumerate}
\item If $\mathsf{gl.dim}_{\mathscr{A}}\mathscr{B}<\infty,$ then
$$\IT(\mathscr{C})\leq\IT(\mathscr{B})\leq\IT(\mathscr{C})+2.$$
Moreover, if $\sf q$ is exact, then $$\IT(\mathscr{C})\leq\IT(\mathscr{B})\leq\IT(\mathscr{C})+1.$$
\item If  $\mathsf{p}$ is exact, then
$$\IT(\mathscr{C})\leq\IT(\mathscr{B})\leq 2\IT(\mathscr{A})+\IT(\mathscr{C})+2.$$
Moreover, if $\mathsf{q}$ is exact, then
$$\IT(\mathscr{C})\leq\IT(\mathscr{B})\leq \IT(\mathscr{A})+\IT(\mathscr{C})+1.$$
\end{enumerate}
\end{corollary}

\section{\bf Applications and examples}

In this section, we apply our results to ring theory, building on Examples \ref{idempotent} and \ref{ex:Morita ring}.

For an Artin algebra $\alg$, the Igusa--Todorov distance  of $\modL$ is denoted by $\IT(\alg)$.
We use $_{\alg}M$ (resp., $M_{\alg}$) to denote a left (resp., right) $\alg$-module $M,$ and denote the projective dimension of $_{\alg}M$ (resp., $M_{\alg}$) by $\mathrm{pd}_{\alg}M$ (resp., $\mathrm{pd}M_{\alg})$.
We denote by $\mathsf{gl.dim}\alg=\sup\{\mathsf{pd}X_{\alg}\mid X\in \mathrm{mod}\text{-}\alg\}.$
Following \cite[Section 8]{P}, we denote by $\mathsf{gl.dim}_{\alg/\alg e\alg}\alg=\sup\{\mathsf{pd}X_{\alg}\mid X\in \mathrm{mod}\text{-}\alg/\alg e\alg\}.$

The following corollary summarizes some results from Section 3, applied to the recollement of module categories  $(\modLe,\modL,\modeLe)$ induced by an idempotent $e=e^{2}$ in an Artin algebra $\alg$.

\begin{corollary} \label{coro:app}
Let $\alg$ be an Artin algebra and let $e=e^{2}$ be an idempotent in $\alg.$ Assume that $e\alg$ is a projective left $e\alg e$-module. Then the following statements hold.
\begin{enumerate}
\item Assume that $\alg/\alg e\alg$ is an $(m,r)$-Igusa--Todorov algebra and $e\alg e$ is an $(n,s)$-Igusa--Todorov algebra. If $\alg/\alg e\alg$ is a projective right $\alg$-module, then $\alg$ is a $(2m+n+2,\max\{r,s\})$-Igusa--Todorov algebra.
\item Assume that $\alg/\alg e\alg$ is an $(m,r)$-Igusa--Todorov algebra and $e\alg e$ is an $(n,s)$-Igusa--Todorov algebra. If $\alg/\alg e\alg$ is a projective left $\alg$-module and $\alg/\alg e\alg$ is a projective right $\alg$-module, then $\alg$ is an $(m+n+1,\max\{r,s\})$-Igusa--Todorov algebra.
\item Assume that $\alg/\alg e\alg$  and $e\alg e$ are syzygy-finite algebras.  If $\alg/\alg e\alg$ is a projective left $\alg$-module and $\alg/\alg e\alg$ is a projective right $\alg$-module, then $\alg$ is an Igusa--Todorov algebra.

\item If $\alg/\alg e\alg$ is a projective right $\alg $-module, then
$$\IT(e\alg e)\leq\IT(\alg)\leq 2\IT(\alg/\alg e\alg)+\IT(e\alg e)+2.$$
 Moreover, if $\alg/\alg e\alg$ is a projective left $\alg$-module, then
$$\IT(e\alg e)\leq\IT(\alg)\leq \IT(\alg/\alg e\alg)+\IT(e\alg e)+1.$$

\item If $\mathsf{gl.dim}_{\alg/\alg e\alg}\alg<\infty$, then
$$\IT(e\alg e)\leq\IT(\alg)\leq\IT(e\alg e)+2.$$
Moreover, if $\alg/\alg e\alg$ is a projective left $\alg$-module, then $$\IT(e\alg e)\leq\IT(\alg)\leq\IT(e\alg e)+1.$$
\end{enumerate}
\end{corollary}

\begin{remark}
 The problem of determining Igusa--Todorov properties for an Artin algebra $\alg$ is generally difficult. The preceding corollary suggests a way forward: select an idempotent $e\in \alg$ for which $e\alg e$ and $\alg/\alg e\alg$ possess ``good" Igusa--Todorov properties and then applying recollements of module categories  to deduce the Igusa--Todorov properties for $\alg$.
\end{remark}

\begin{example} \rm
We provide an instance for Corollary \ref{coro:app}(4).
For each quiver in this example, the composition of two arrows $\alpha$ and $\beta$ (the ending point $\mathfrak{t}(\alpha)$ of $\alpha$ and the starting point $\mathfrak{s}(\beta)$ of $\beta$ coincide) is written as $\alpha\beta$, c.f. \cite{ASS2006}.
Take $\alg_{\mathrm{I}}$ the bound quiver algebra given by
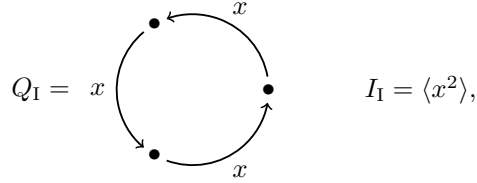
\begin{center}
\begin{tikzpicture}
\draw (-2,0) node{$Q_{\mathrm{I}}=$};
\draw (3,0) node{$I_{\mathrm{I}} = \langle x^2 \rangle$, };
\draw[->][line width=0.7pt][rotate = 10] (1,0) arc(0:100:1);
\draw[->][line width=0.7pt][rotate =130] (1,0) arc(0:100:1);
\draw[->][line width=0.7pt][rotate =250] (1,0) arc(0:100:1);
\foreach \x in {0,120,240}
\draw [rotate =\x] (1,0) node{$\bullet$};
\foreach \x in {0,120,240}
\draw [rotate =\x+60] (1.25,0) node{$x$};
\end{tikzpicture}
\end{center}
and $\alg_{\mathrm{II}}$ the hereditary quiver algebra given by $\xymatrix{\bullet \ar[r] & \bullet}$.
Let $\alg := \alg_{\mathrm{I}} \otimes_{K} \alg_{\mathrm{II}}$, where $K$ is an algebraically closed field.
Then the bound quiver of $\alg$ is
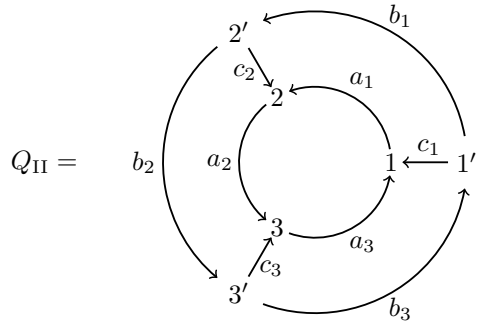
\begin{center}
\begin{tikzpicture}
\draw (-3,0) node[left]{$Q_{\mathrm{II}}=$};
\draw[->][line width=0.7pt][rotate = 10] (1,0) arc(0:100:1);
\draw[->][line width=0.7pt][rotate =130] (1,0) arc(0:100:1);
\draw[->][line width=0.7pt][rotate =250] (1,0) arc(0:100:1);
\draw[->][line width=0.7pt][rotate = 10] (2,0) arc(0:100:2);
\draw[->][line width=0.7pt][rotate =130] (2,0) arc(0:100:2);
\draw[->][line width=0.7pt][rotate =250] (2,0) arc(0:100:2);
\draw [rotate =  0] (1,0) node{$1$};
\draw [rotate =120] (1,0) node{$2$};
\draw [rotate =240] (1,0) node{$3$};
\draw [rotate =  0] (2,0) node{$1'$};
\draw [rotate =120] (2,0) node{$2'$};
\draw [rotate =240] (2,0) node{$3'$};
\draw [rotate =  0+60] (1.25,0) node{$a_1$};
\draw [rotate =120+60] (1.25,0) node{$a_2$};
\draw [rotate =240+60] (1.25,0) node{$a_3$};
\draw [rotate =  0+60] (2.25,0) node{$b_1$};
\draw [rotate =120+60] (2.25,0) node{$b_2$};
\draw [rotate =240+60] (2.25,0) node{$b_3$};
\draw [rotate =  0]     (1.5,.2) node{$c_1$};
\draw [rotate =120]     (1.5,.2) node{$c_2$};
\draw [rotate =240]     (1.5,.2) node{$c_3$};
\draw [rotate =  0][line width=0.7pt][->] (1.75,0) -- (1.15,0);
\draw [rotate =120][line width=0.7pt][->] (1.75,0) -- (1.15,0);
\draw [rotate =240][line width=0.7pt][->] (1.75,0) -- (1.15,0);
\draw (0,-2.5) node[below]{$I_{\mathrm{II}}=\langle a_1a_2, a_2a_3, a_3a_1, b_1b_2, b_2b_3, b_3b_1,
b_1c_2-c_1a_1, b_2c_3-c_2a_2, b_3c_1-c_3a_3 \rangle.$};
\end{tikzpicture}
\end{center}
Let $e = e_{1'}+e_{2'}+e_{3'}$ (for any vertex $v$, $e_v$ is the primitive idempotent corresponding to $v$). Then
\[ e\alg = Ke_1 + Ke_2 + Ke_3 + Ka_1 + Ka_2 + Ka_3 \]
is both a left $e\alg e$-module and a right $\alg$-module.
Notice that $e$ is the identity element of $e\alg e$. Then for any left $e\alg e$-module $M$,
the decomposition $M = eM = (e_1+e_2+e_3)M = e_1M \oplus e_2M \oplus e_3M$ provides a quiver representation of $M$ as
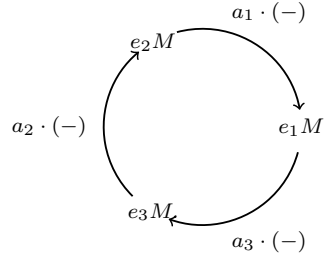
\begin{center}
\begin{tikzpicture}[scale=1.3]
\draw[<-][line width=0.7pt][rotate = 10] (1,0) arc(0:95:1);
\draw[<-][line width=0.7pt][rotate =130] (1,0) arc(0:95:1);
\draw[<-][line width=0.7pt][rotate =250] (1,0) arc(0:95:1);
\draw [rotate =  0] (1,0) node{\footnotesize$e_1M$};
\draw [rotate =120] (1,0) node{\footnotesize$e_2M$};
\draw [rotate =240] (1,0) node{\footnotesize$e_3M.$};
\draw [rotate =  0+60] (1.35,0) node{\footnotesize$a_1\cdot(-)$};
\draw [rotate =120+60] (1.55,0) node{\footnotesize$a_2\cdot(-)$};
\draw [rotate =240+60] (1.35,0) node{\footnotesize$a_3\cdot(-)$};
\end{tikzpicture}
\end{center}
Here, the quiver of $e\alg e$ equals to $Q_{\mathrm{I}}$, and we need to reverse the direction of the arrow in the quiver representation of left $e\alg e$-modules $M$.
Thus, one can check that the quiver representation of the left $e\alg e$-module $e\alg$ is
\begin{center}
\begin{tikzpicture}[scale=1.3]
\draw (1.5,0) node{$\cong$};
\draw[<-][line width=0.7pt][rotate = 10] (1,0) arc(0:95:1);
\draw[<-][line width=0.7pt][rotate =130] (1,0) arc(0:95:1);
\draw[<-][line width=0.7pt][rotate =250] (1,0) arc(0:95:1);
\draw [rotate =  0] (1,0) node{\footnotesize$e_{1'}\alg$};
\draw [rotate =120] (1,0) node{\footnotesize$e_{2'}\alg$};
\draw [rotate =240] (1,0) node{\footnotesize$e_{3'}\alg$};
\draw [rotate =  0+60] (1.35,0) node{\footnotesize$b_1\cdot(-)$};
\draw [rotate =120+60] (1.55,0) node{\footnotesize$b_2\cdot(-)$};
\draw [rotate =240+60] (1.35,0) node{\footnotesize$b_3\cdot(-)$};
\end{tikzpicture}
\begin{tikzpicture}[scale=1.3]
\draw[<-][line width=0.7pt][rotate = 10] (1,0) arc(0:95:1);
\draw[<-][line width=0.7pt][rotate =130] (1,0) arc(0:95:1);
\draw[<-][line width=0.7pt][rotate =250] (1,0) arc(0:95:1);
\draw [rotate =  0] (1,0) node{$K^4$,};
\draw [rotate =120] (1,0) node{$K^4$};
\draw [rotate =240] (1,0) node{$K^4$};
\draw [rotate =  0+60] (1.45,0)
  node{\footnotesize
  $\left[ \begin{smallmatrix}
    0 & 0 & 0 & 0 \\
    1 & 0 & 0 & 0 \\
    0 & 1 & 0 & 0 \\
    0 & 0 & 0 & 0
  \end{smallmatrix} \right]$};
\draw [rotate =120+60] (1.45,0)
  node{\footnotesize
  $\left[ \begin{smallmatrix}
    0 & 0 & 0 & 0 \\
    1 & 0 & 0 & 0 \\
    0 & 1 & 0 & 0 \\
    0 & 0 & 0 & 0
  \end{smallmatrix} \right]$};
\draw [rotate =240+60] (1.45,0)
  node{\footnotesize
  $\left[ \begin{smallmatrix}
    0 & 0 & 0 & 0 \\
    1 & 0 & 0 & 0 \\
    0 & 1 & 0 & 0 \\
    0 & 0 & 0 & 0
  \end{smallmatrix} \right]$};
\end{tikzpicture}
\end{center}
which admits that
\[ {_{e\alg e}(e\alg)} \cong
  {_{e\alg e}P(1')^{\oplus 2}}
  \oplus {_{e\alg e}P(2')^{\oplus 2}}
  \oplus {_{e\alg e}P(3')^{\oplus 2}} \]
is a projective left $e\alg e$-module.

Second, $\alg/\alg e \alg = Ke_{1}+Ke_{2}+Ke_{3}+Ka_1+Ka_2+Ka_3 + \langle e\rangle$ is a right $\alg$-module, we have
\begin{align*}
& (\alg/\alg e \alg)e_{1'} = (\alg/\alg e \alg)e_{2'} = (\alg/\alg e \alg)e_{3'} = 0; \\
& (\alg/\alg e \alg)e_{1} = Ke_{1}+Kb_3 \cong_K K^2, \\
& (\alg/\alg e \alg)e_{2} = Ke_{2}+Kb_1 \cong_K K^2, \\
& (\alg/\alg e \alg)e_{3} = Ke_{3}+Kb_2 \cong_K K^2,
\end{align*}
where ``$\cong_K$'' presents the isomorphism of $K$-linear spaces,
and then, the quiver representation of $\alg/\alg e \alg$, as a right $\alg$-module, is
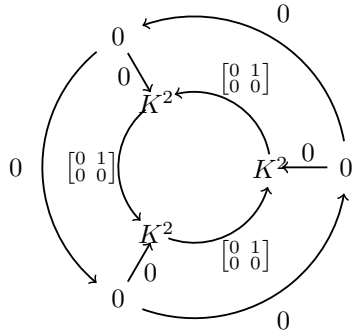
\begin{center}
\begin{tikzpicture}
\draw[->][line width=0.7pt][rotate = 10] (1,0) arc(0: 95:1);
\draw[->][line width=0.7pt][rotate =130] (1,0) arc(0: 95:1);
\draw[->][line width=0.7pt][rotate =250] (1,0) arc(0: 95:1);
\draw[->][line width=0.7pt][rotate = 10] (2,0) arc(0:100:2);
\draw[->][line width=0.7pt][rotate =130] (2,0) arc(0:100:2);
\draw[->][line width=0.7pt][rotate =250] (2,0) arc(0:100:2);
\draw [rotate =  0] (1,0) node{$K^2$};
\draw [rotate =120] (1,0) node{$K^2$};
\draw [rotate =240] (1,0) node{$K^2$};
\draw [rotate =  0] (2,0) node{$0$};
\draw [rotate =120] (2,0) node{$0$};
\draw [rotate =240] (2,0) node{$0$};
\draw [rotate =  0+60] (1.35,0) node{$\big[{^0_0\ ^1_0}\big]$};
\draw [rotate =120+60] (1.35,0) node{$\big[{^0_0\ ^1_0}\big]$};
\draw [rotate =240+60] (1.35,0) node{$\big[{^0_0\ ^1_0}\big]$};
\draw [rotate =  0+60] (2.35,0) node{$0$};
\draw [rotate =120+60] (2.35,0) node{$0$};
\draw [rotate =240+60] (2.35,0) node{$0$};
\draw [rotate =  0]     (1.5,.2) node{$0$};
\draw [rotate =120]     (1.5,.2) node{$0$};
\draw [rotate =240]     (1.5,.2) node{$0$};
\draw [rotate =  0][line width=0.7pt][->] (1.75,0) -- (1.15,0);
\draw [rotate =120][line width=0.7pt][->] (1.75,0) -- (1.15,0);
\draw [rotate =240][line width=0.7pt][->] (1.75,0) -- (1.15,0);
\end{tikzpicture}
\end{center}
Thus, we get that
\[ {(\alg/\alg e \alg)_{\alg}} \cong {P(1)_{\alg}} \oplus {P(2)_{\alg}} \oplus {P(3)_{\alg}} \]
is projective. Then we obtain
\[\IT(e\alg e)\leq\IT(\alg)\leq 2\IT(\alg/\alg e\alg)+\IT(e\alg e)+2\]
by using Corollary \ref{coro:app}(4).
%
%
Furthermore, since $\alg/\alg e \alg \cong \alg_{\mathrm{I}}$ and $e \alg e \cong \alg_{\mathrm{I}}$
are syzygy-finite since they are monomial algebras by using \cite[Theorem A (1)]{Hui1992},
we have $\IT(e\alg e)=0$ and $\IT(\alg/\langle e\rangle)=0$.
Then $\IT(\alg)\leq 2$. Notice that $\alg$ is representation-finite
(to be more precise, one can check that the number of indecomposable right $\alg$-modules is $27$ (up to isomorphisms) by using Auslander--Reiten quiver), it follows that $\IT(\alg)=0$, this fact coincides with $\IT(\alg)\leq 2$.
\end{example}

\begin{example} \rm
We provide an instance for Corollary \ref{coro:app}(5).
Take $\alg=KQ/I$ given by
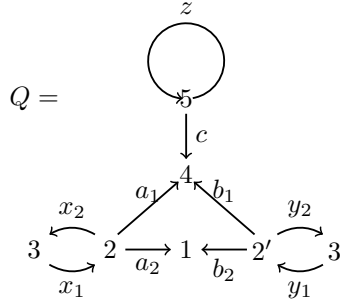
\begin{center}
\begin{tikzpicture}
\draw[->][line width=0.7pt] ( 0.8, 0.0) -- ( 0.2, 0.0);
\draw[->][line width=0.7pt] [shift={(0, 0.2)}] ( 1.2, 0.0) to[out= 30, in=150] ( 1.8, 0.0);
\draw[->][line width=0.7pt] [shift={(0,-0.2)}] ( 1.8, 0.0) to[out=210, in=-30] ( 1.2, 0.0);
\draw[->][line width=0.7pt] (-0.8, 0.0) -- (-0.2, 0.0);
\draw[->][line width=0.7pt] [shift={(0, 0.2)}] (-1.2, 0.0) to[out=150, in= 30] (-1.8, 0.0);
\draw[->][line width=0.7pt] [shift={(0,-0.2)}] (-1.8, 0.0) to[out=-30, in=210] (-1.2, 0.0);
\draw[<-][line width=0.7pt] ( 0.0, 1.2) -- ( 0.0, 1.8);
\draw ( 0.0, 0.0) node{$1$}
      (-1.0, 0.0) node{$2$}
      (-2.0, 0.0) node{$3$}
      ( 1.0, 0.0) node{$2'$}
      ( 2.0, 0.0) node{$3'$}
      ( 0.0, 1.0) node{$4$}
      ( 0.0, 2.0) node{$5$};
\draw[<-][line width=0.7pt] (-0.1, 0.9) -- (-0.9, 0.2);
\draw[<-][line width=0.7pt] ( 0.1, 0.9) -- ( 0.9, 0.2);
\draw ( 0.5, 0.0) node[below]{$b_2$};
\draw (-0.5, 0.0) node[below]{$a_2$};
\draw (-0.5, 0.5) node[above]{$a_1$};
\draw ( 0.5, 0.5) node[above]{$b_1$};
\draw ( 0.0, 1.5) node[right]{$c$};
\draw (-1.5, 0.3) node[above]{$x_2$};
\draw (-1.5,-0.3) node[below]{$x_1$};
\draw ( 1.5, 0.3) node[above]{$y_2$};
\draw ( 1.5,-0.3) node[below]{$y_1$};
\draw[->][line width=0.7pt][rotate around={10:(0,2.5)}] ( 0.0, 2.0) arc (-90:255:0.5);
\draw ( 0.0, 3.0) node[above]{$z$};
\draw (-2,2) node{$Q=$};
\end{tikzpicture}

and $I = \langle x_1x_2, y_1y_2, z^2, x_1a_1, y_1b_1 \rangle.$
\end{center}
Then all indecomposable right $\alg$-modules are
\begin{align*}
  & P(1) = (1)_{\alg}  \cong S(1),
 && P(2) = \left(\begin{smallmatrix} &2& \\ 3&4&1 \\ 2&& \\ 1&& \end{smallmatrix} \right)_{\alg},
 && P(3) = \left(\begin{smallmatrix} 3 \\ 2 \\ 1 \end{smallmatrix} \right)_{\alg},
 && P(4) = \left(4\right)_{\alg} \cong S(4),
\end{align*}
\begin{align*}
  & P(2') = \left(\begin{smallmatrix} &2'& \\ 1&4&3' \\ &&2' \\ &&1 \end{smallmatrix} \right)_{\alg},
 && P(3') = \left(\begin{smallmatrix} 3' \\ 2' \\ 1' \end{smallmatrix} \right)_{\alg},
 && P(5) = \left(\begin{smallmatrix} &5& \\ 4&&5 \\ &&4 \end{smallmatrix} \right)_{\alg}
 &&
\end{align*}
Thus, one can check that $\pd S(2) = \pd S(2') = 1$, $\pd S(3) = \pd S(3') = 2$, and $\pd S(5)=\infty$,
and then we have $\mathsf{gl.dim}\alg = \infty$.
Take $e=e_4$, then we have $e\alg e$ is simple, it follows that a trivial case that $e\alg$, as a left $e\alg e$-module, is isomorphic to the projective left $e\alg e$-module $(eAe)^{\oplus \mathrm{dim} e\alg}$.
In this case, since $4$ is the source of the quiver $Q$ and the classification of all indecomposable right $A/\langle e\rangle$-module can be described by using strings and bands (c.f. \cite[Section 3]{BR1987}), it is easy to check that the algebra
\[A/\langle e\rangle = K\big(\xymatrix{
  3 \ar@/_.5pc/[r]_{x_1} & 2 \ar@/_.5pc/[l]_{x_2}
& 1 \ar@{<-}[r]_{b_2} \ar@{<-}[l]^{a_2}
& 2' \ar@/^.5pc/[r]^{y_2} \ar@/_.5pc/[r]_{y_1}
& 3
}\big)/\langle x_1x_2, y_1y_2 \rangle ~ \times ~ K[z]/\langle z^2\rangle \]
satisfies $\mathsf{gl.dim}_{A/\langle e\rangle} A = 2 < \infty$. Then we have
\[ 0=\IT(e\alg e) \leq \IT(\alg) \leq \IT(e\alg e)+2 = 2 \]
by Corollary \ref{coro:app}(5). Moreover, $\alg$ is representation-infinite (band modules exists, c.f. \cite{BR1987}), it follows that $\IT(\alg) \geq 1$, and so $1 \leq \IT(\alg) \leq 2$.

Furthermore, $A/\langle e\rangle$ is a left $A$-module, and we have
\begin{align*}
  & e_1(A/\langle e\rangle) = Ke_1+I, \\
  & e_2(A/\langle e\rangle) = Ke_2+Ka_2+Kx_2+Kx_2x_1+Kx_2x_1a_2+I \cong_K K^5, \\
  & e_3(A/\langle e\rangle) = Ke_3+Kx_1+Kx_1a_2 + I \cong_K K^3, \\
  & e_{2'}(A/\langle e\rangle) = Ke_{2'}+Kb_2+Ky_2+Ky_2y_1+Ky_2y_1b_2 + I \cong_K K^5, \\
  & e_{3'}(A/\langle e\rangle) = Ke_{3'}+Ky_1+Ky_1b_2 + I \cong_K K^3, \\
  & e_{4}(A/\langle e\rangle) = 0, \\
  & e_{5}(A/\langle e\rangle) = Ke_5+Kz + \langle z^2\rangle \cong K^2
\end{align*}
(``$\cong_K$'' presents the isomorphism of $K$-linear spaces).
Then the quiver representation of $A/\langle e\rangle$ is of the following form
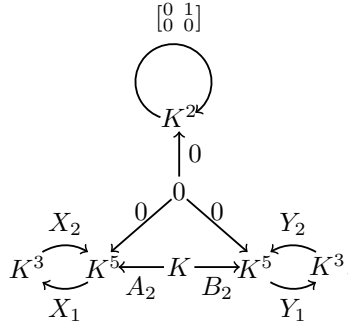
\begin{center}
\begin{tikzpicture}
\draw[<-][line width=0.7pt] ( 0.8, 0.0) -- ( 0.2, 0.0);
\draw[<-][line width=0.7pt] [shift={(0, 0.2)}] ( 1.2, 0.0) to[out= 30, in=150] ( 1.8, 0.0);
\draw[<-][line width=0.7pt] [shift={(0,-0.2)}] ( 1.8, 0.0) to[out=210, in=-30] ( 1.2, 0.0);
\draw[<-][line width=0.7pt] (-0.8, 0.0) -- (-0.2, 0.0);
\draw[<-][line width=0.7pt] [shift={(0, 0.2)}] (-1.2, 0.0) to[out=150, in= 30] (-1.8, 0.0);
\draw[<-][line width=0.7pt] [shift={(0,-0.2)}] (-1.8, 0.0) to[out=-30, in=210] (-1.2, 0.0);
\draw[->][line width=0.7pt] ( 0.0, 1.2) -- ( 0.0, 1.8);
\draw ( 0.0, 0.0) node{$K$}
      (-1.0, 0.0) node{$K^5$}
      (-2.0, 0.0) node{$K^3$}
      ( 1.0, 0.0) node{$K^5$}
      ( 2.0, 0.0) node{$K^3,$}
      ( 0.0, 1.0) node{$0$}
      ( 0.0, 2.0) node{$K^2$};
\draw[->][line width=0.7pt] (-0.1, 0.9) -- (-0.9, 0.2);
\draw[->][line width=0.7pt] ( 0.1, 0.9) -- ( 0.9, 0.2);
\draw ( 0.5, 0.0) node[below]{$B_2$};
\draw (-0.5, 0.0) node[below]{$A_2$};
\draw (-0.5, 0.5) node[above]{$0$};
\draw ( 0.5, 0.5) node[above]{$0$};
\draw ( 0.0, 1.5) node[right]{$0$};
\draw (-1.5, 0.3) node[above]{$X_2$};
\draw (-1.5,-0.3) node[below]{$X_1$};
\draw ( 1.5, 0.3) node[above]{$Y_2$};
\draw ( 1.5,-0.3) node[below]{$Y_1$};
\draw[<-][line width=0.7pt][rotate around={25:(0,2.5)}] ( 0.0, 2.0) arc (-80:225:0.5);
\draw ( 0.0, 3.0) node[above]{$[{_0^0}\ {_0^1}]$};
\end{tikzpicture}
\end{center}
where
\begin{align*}
& X_1=Y_1 = \left[
  {\begin{smallmatrix}
   0 & 0 & 0 & 0 & 0 \\
   1 & 0 & 0 & 0 & 0 \\
   0 & 0 & 0 & 0 & 1
  \end{smallmatrix}}
  \right],
&& X_2=Y_2 = \left[
  {\begin{smallmatrix}
   0 & 0 & 0 \\
   1 & 0 & 0 \\
   0 & 1 & 0 \\
   0 & 0 & 1 \\
   0 & 0 & 0
  \end{smallmatrix}}
  \right],
&& A_2=B_2 =  \left[
  {\begin{smallmatrix}
   0 \\ 0 \\ 0 \\ 0 \\ 1
  \end{smallmatrix}}\right].
\end{align*}
Therefore, $A/\langle e\rangle \cong {_{\alg}P(1)}\oplus {_{\alg}P(2)} \oplus {_{\alg}P(3)} \oplus {_{\alg}P(2')} \oplus {_{\alg}P(3)} \oplus {_{\alg}P(5)}$ is a left projective $\alg$-module.
Then Corollary \ref{coro:app} (5) yields $\IT(e\alg e)\leq\IT(\alg)\leq\IT(e\alg e)+1$,
i.e., $\IT(\alg)=1$. This coincides with the fact that Igusa--Todorov distance of representation-infinite monomial algebra is $1$.
\end{example}

Next we consider the Morita context ring $\alg =\left(\begin{smallmatrix}  A & {_{A}}N_{B}\\  {_{B}}M_{A} & B \\\end{smallmatrix}\right)$
which is an Artin algebra and the bimodule morphisms are zero.
By Example \ref{ex:Morita ring}, there is a recollement of abelian categories
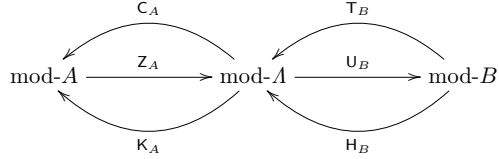
\begin{align*}
\scalebox{0.85}{\xymatrixcolsep{2pc}\xymatrix{
			\modA\ar[rr]^{\mathsf{Z}_{A}} && \modL\ar@/_2.0pc/[ll]_{\mathsf {C}_{A}}\ar@/^2.0pc/[ll]^{\mathsf{K}_{A}}\ar[rr]^{\mathsf {U}_{B}} && \modB \ar@/_2.0pc/[ll]_{\mathsf{T}_{B}}\ar@/^2.0pc/[ll]^{\mathsf{H}_{B}} } }
\end{align*}
which is induced by the idempotent element $\left(\begin{smallmatrix}  0 & 0\\  0 & 1 \\\end{smallmatrix}\right)$ in $\alg$.
It follows from Example \ref{idempotent} that $$\mathsf{C}_{A}\simeq-\otimes_{\alg}(A,0,0,0) ~~~~~~~\text{and }~~~~~~~ \mathsf{K}_{A}\simeq\Hom_{\alg}((M,B,0,1),-)\oplus \Hom_{\alg}((0,N,0,0),-).$$
Hence, we have the following facts:
\begin{enumerate}
\item The functor $\mathsf{C}_{A}$ is exact if and only if $M=0$.
\item The functor $\mathsf{K}_{A}$ is exact if and only if $N_{B}$ is projective and $N\otimes_{B}M=0.$
\end{enumerate}

\begin{lemma}[{\cite[Corollary 5.7]{GP}}]\label{lem:Morita ring}
Let $\alg =\left(\begin{smallmatrix}  A & {_{A}}N_{B}\\  {_{B}}M_{A} & B \\\end{smallmatrix}\right)$
be a Morita context ring which is an Artin algebra and the bimodule morphisms are zero.  Then for a right $A$-module $X$ we have
$$\mathsf{pd}(X,0,0,0)_{\alg}\leq \mathsf{pd}X_{A}+\mathsf{pd}(0,N,0,0)_{\alg}.$$
\end{lemma}

\begin{definition}(\cite[Definition 4.4]{CL})
	Let $R$ be a ring. For an $R$-bimodule $M$ consider the following conditions:
\begin{enumerate}
\item $\mathrm{pd}_{R}M<\infty$,
\item $\mathrm{pd}{M_{R}}<\infty$, and
\item $\mathrm{Tor}_{i}^{R}(M^{\oo j},M)=0$ for all $i,j\geq1.$
\end{enumerate}
If $M$ satisfies (1) and (3), then it is called
 {\bf left perfect}. If $M$ satisfies (1) and (3), then it is called
 {\bf right perfect}.  If $M$ satisfies (1), (2) and (3), it is called {\bf perfect}
	
The $R$-module $M$ is called {\bf nilpotent} if $M^{\oo n}=0$ for some $n.$
\end{definition}

\begin{lemma}[{\cite[Proposition 5.4]{KP}}]\label{lem:Morita ring1}
Let $\alg =\left(\begin{smallmatrix}  A & {_{A}}N_{A}\\  {_{A}}N_{A} & A \\\end{smallmatrix}\right)$
be a Morita context ring which is an Artin algebra and the bimodule morphisms are zero. Assume that $N$ is nilpotent and right perfect.
 Then we have
$\mathsf{pd}{(0,N,0,0)_{\alg}}< \infty.$
\end{lemma}

\begin{corollary}
Let $\alg =\left(\begin{smallmatrix}  A & {_{A}}N_{A}\\  {_{A}}N_{A} & A \\\end{smallmatrix}\right)$
be a Morita context ring which is an Artin algebra and the bimodule morphisms are zero. Assume that $N$ is a projective left $A$-module and $\mathsf{gl.dim}A<\infty$. If $N$ is nilpotent and right perfect, then
$$\IT(A)\leq\IT(\alg)\leq\IT(A)+2.$$
\end{corollary}
\begin{proof}
Since $N$ is nilpotent and left perfect, it follows from Lemmas \ref{lem:Morita ring} and \ref{lem:Morita ring1}
that $$\sup\{\mathsf{pd}(X,0,0,0)_{\alg}\mid X\in \modA\}\leq\mathsf{gl.dim}A+\mathsf{pd}(0,N,0,0)_{\alg}<\infty.$$
Therefore by Corollary \ref{corollary:main}(1) we get inequalities $\IT(A)\leq\IT(\alg)\leq\IT(A)+2.$
\end{proof}

We recall the following notion introduced in \cite{GP}.

\begin{definition}[{\cite[Definition 5.2]{GP}}]
If $(P,0,0,0)$ is a projective right $\alg$-module for some right $A$-module $P$, then
$(P,0,0,0)$ is called an $A$-{\bf tight projective} $\alg$-module. We say that a right $\alg$-module $(X,0,0,0)$ has an {\bf $A$-tight projective $\alg$-resolution} if $(X,0,0,0)$ has a projective $\alg$-resolution in which each projective $\alg$-module is $A$-tight. We have the similar definition for {\bf $B$-tight projective $\alg$-modules} and {\bf $B$-tight projective $\alg$-resolutions}.
\end{definition}

\begin{remark}[{\cite{GP}}]\label{rem:tight}
(1) A right $\alg$-module $(P,0,0,0)$ is an $A$-tight projective $\alg$-module if and only if $P$ is a projective right $A$-module and $P\otimes_{A}N=0.$

(2) A right $\alg$-module $(0,Q,0,0)$ is a $B$-tight projective $\alg$-module if and only if $Q$ is a projective right $B$-module and $Q\otimes_{B}M=0.$

\end{remark}
Recall from \cite[Section 5]{GP} that we write that  a right $B$-module $Y$ has a $B$-tight projective $\alg$-resolution meaning that the object $(0,Y,0,0)$, as a right $\alg$-module, has a $B$-tight projective $\alg$-resolution.

\begin{lemma}[{\cite[Proposition 5.8]{GP}}]\label{lem:Morita ring2}
Let $\alg =\left(\begin{smallmatrix}  A & {_{A}}N_{B}\\  {_{B}}M_{A} & B \\\end{smallmatrix}\right)$
be a Morita context ring which is an Artin algebra and the bimodule morphisms are zero. If right $B$-module $N$ has a $B$-tight projective $\alg$-resolution, then for a right $A$-module $X$ we have
$$\mathsf{pd}(X,0,0,0)_{\alg}\leq \mathsf{pd}X_{A}+\mathsf{pd}N_{B}.$$
\end{lemma}

\begin{corollary}\label{application:Morita ring1}
Let $\alg =\left(\begin{smallmatrix}  A & {_{A}}N_{B}\\  {_{B}}M_{A} & B \\\end{smallmatrix}\right)$
be a Morita context ring which is an Artin algebra and the bimodule morphisms are zero. Assume that $M$ is a projective left $B$-module and $\mathsf{gl.dim}A<\infty$. If right $B$-module $N$ has a $B$-tight projective $\alg$-resolution and $\mathsf{pd}N_{B}<\infty$
, then
$$\IT(A)\leq\IT(\alg)\leq\IT(A)+2.$$
\end{corollary}
\begin{proof}
Since  $N$ has a $B$-tight projective $\alg$-resolution, it follows from Lemma \ref{lem:Morita ring2}
that $$\sup\{\mathsf{pd}(X,0,0,0)_{\alg}\mid X\in \modA\}\leq\mathsf{gl.dim}A+\mathsf{pd}N_{B}<\infty.$$
Therefore by Corollary \ref{corollary:main}(1) we get inequalities $\IT(A)\leq\IT(\alg)\leq\IT(A)+2.$
\end{proof}

\begin{corollary}\label{application:Morita ring}
Let $\alg =\left(\begin{smallmatrix}  A & {_{A}}N_{B}\\  {_{B}}M_{A} & B \\\end{smallmatrix}\right)$
	be a Morita context ring which is an Artin algebra and the bimodule morphisms are zero. Assume that $M$ is a projective left $B$-module. Then the following statements hold:
\begin{enumerate}
\item Assume that $A$ is an $(m,r)$-Igusa--Todorov algebra algebra and $B$ is an $(n,s)$-Igusa--Todorov algebra. If $N$ is a projective right $B$-module and $N\otimes_{B}M=0$, then $\alg$ is an $(2m+n+2,\max\{r,s\})$-Igusa--Todorov algebra.

\item If $N$ is a projective right $B$-module and $N\otimes_{B}M=0$, then
$$\IT(B)\leq\IT(\alg)\leq 2\IT(A)+\IT(B)+2.$$
\end{enumerate}
\end{corollary}
\begin{proof}
Since $N$ is a projective right $B$-module and $N\otimes_{B}M=0$, it follows from Remark \ref{rem:tight} that $(0,N,0,0)$ is a $B$-tight projective $\alg$-module and so $\mathsf{K}_{A}$ is exact.
Hence this corollary follows from Example \ref{ex:Morita ring}, Theorem \ref{main theorem}(1) and Corollary \ref{corollary:main}(2).
\end{proof}

By Corollary \ref{corollary:main} and Example \ref{ex:Morita ring}, we get the following result.
\begin{corollary}\label{corollary:triangular matrix ring}
Let $\alg =\left(\begin{smallmatrix}  A & {_{A}}N_{B}\\  0 & B \\\end{smallmatrix}\right)$
	be a triangular matrix Artin algebra. Assume that $N$ is a projective right  $B$-module. Then the following statements hold:
\begin{enumerate}
\item If $A$ is an $(m,r)$-Igusa--Todorov algebra algebra and $B$ is an $(n,s)$-Igusa--Todorov algebra, then $\alg$ is a $(m+n+1,\max\{r,s\})$-Igusa--Todorov algebra.

\item  If $\mathsf{gl.dim}A<\infty$, then
$$\IT(B)\leq\IT(\alg)\leq\IT(B)+1.$$
\item There are inequalities
$$\IT(B)\leq\IT(\alg)\leq \IT(A)+\IT(B)+1.$$
In particular, if both $A$ and $B$ are syzygy-finite algebras, then so is $\alg.$

\end{enumerate}
\end{corollary}

\begin{remark}
One can compare Corollary \ref{corollary:triangular matrix ring} with \cite[Corollary 6.18]{ZZ1}. However, our methods used here are different from those in \cite{ZZ1}.
\end{remark}

\section*{Acknowledgments}
\noindent
The authors would like to thank Yongyun Qin for questions and corrections that improved the paper.

\def\Up{\mathrm{U}}
\def\Down{\mathrm{D}}
\def\e{\varepsilon}
\def\modcat{\mathrm{mod}\text{-}}
\def\I{\mathrm{I}}
\def\II{\mathrm{II}}
\def\rad{\mathrm{rad}}
\def\soc{\mathrm{soc}}
\def\lineardim{\mathrm{dim}_{k}}

\bigskip

\noindent\textbf{Peiru Yang} \\
School of Mathematics and Statistics, Northeast Normal University, Changchun, 730070, Jilin Province, P. R. China. \\
E-mail: \textsf{yangprlife@163.com}

\noindent\textbf{Yajun Ma} \\
School of Mathematics and Physics, Lanzhou Jiaotong University, Lanzhou, 730070, Gansu Province, P. R. China. \\
E-mail: \textsf{yjma@mail.lzjtu.cn}

\noindent\textbf{Yu-Zhe Liu} \\
School of Mathematics and statistics, Guizhou University, 550025 Guiyang, Guizhou, P. R. China. \\
E-mail: \textsf{liuyz@gzu.edu.cn / yzliu3@163.com}

\end{document}